\documentclass[12pt]{amsart}

\usepackage{amssymb}
\usepackage{amsmath}
\usepackage{graphicx} 
\usepackage{enumerate}
\usepackage{color}
\usepackage[pdftex]{hyperref}

\newcommand{\C}{\mathbb C}

\newcommand{\R}{\mathbb R}

\newcommand{\Z}{\mathbb Z}

\renewcommand{\P}{\mathbb P}

\newcommand{\cH}{\mathcal H}
\newcommand{\cI}{\mathcal I}
\newcommand{\cJ}{\mathcal J}

\newcommand{\cM}{\mathcal M}
\newcommand{\cV}{\mathcal V}
\newcommand{\X}{\mathfrak X}
\newcommand{\bF}{\mathbf F}

\newcommand{\na}{\nabla}
\newcommand{\pa}{\partial}

\renewcommand{\a}{\alpha} 
\renewcommand{\b}{\beta} 

\newcommand{\G}{\varGamma}
\newcommand{\De}{\mathit{\Delta}}
\newcommand{\de}{\delta}

\newcommand{\f}{\varphi}

\newcommand{\w}{\omega}
\newcommand{\W}{\mathit{\Omega}}
\newcommand{\la}{\langle}
\newcommand{\ra}{\rangle}
\newcommand{\dd}{\mathrm{d}}
\newcommand{\dt}{\mathrm{d}_t}
\newcommand{\dx}{\mathrm{d}_x}
\newcommand{\tx}{\widetilde{x}}
\newcommand{\tr}{\;^t}

\newcommand{\bu}{\bullet}
\newcommand{\diag}{\mathrm{diag}}

\newcommand{\under}[3]{
  {}_{#1} #2 _{#3}
}

\newcommand{\DS}{\displaystyle }
\newcommand{\si}{\sigma}
\newcommand{\we}{\wedge}
\newcommand{\tpi}{2\pi \sqrt{-1}}
\newcommand{\bS}{\mathbf{S}}
\newcommand{\naa}{\nabla^{\alpha}}
\newcommand{\Fax}{F(\alpha ;x)}
\newcommand{\bFax}{\bF (\alpha ;x)}
\newcommand{\Psiax}{\Psi (\alpha ;x)}
\newcommand{\Psiij}{\Psi_{ij} (\alpha ;x)}
\newcommand{\MJa}{M_J (\alpha)}
\newcommand{\MJij}{M_J^{ij} (\alpha)}
\newcommand{\cMJ}{\cM_J(\alpha)}
\newcommand{\conti}[1]{U_{#1}}
\newcommand{\Conti}[1]{\mathcal{U}_{#1}}
\newcommand{\cohom}{V}
\newcommand{\cohomi}{{V^{(i)}}}
\newcommand{\Omd}{\W^{\bu}}

\newtheorem{theorem}{Theorem}[section]
\newtheorem{proposition}[theorem]{Proposition}
\newtheorem{lemma}[theorem]{Lemma}
\newtheorem{cor}[theorem]{Corollary}
\newtheorem{fact}[theorem]{Fact}
\theoremstyle{definition}
\newtheorem{remark}[theorem]{Remark}

\newtheorem{exam}[theorem]{Example}

\newtheorem{algorithm}[theorem]{Algorithm}
\newtheorem{notation}[theorem]{Notation}
\newtheorem*{acknowledgements}{Acknowledgements}
\makeatletter
 
 \@addtoreset{equation}{section}
\makeatother   
\title
[Pfaffian and contiguity relations]
{Pfaffian equations and contiguity relations
  of the hypergeometric function of type
  $(k+1, k+n+2)$ and their applications} 

\author[Y. Goto]{Yoshiaki GOTO}
\address[Goto]{
Department of Mathematics, 
Graduate School of Science,
Kobe University, 
Kobe 657-8501, Japan}
\email{y-goto@math.kobe-u.ac.jp}
\author[K. Matsumoto]{Keiji Matsumoto}
\address[Matsumoto]{
Department of Mathematics\\
Hokkaido University\\
Sapporo 060-0810, Japan
}
\email{matsu@math.sci.hokudai.ac.jp}

\date{\today}

\begin{document}
\maketitle
\begin{abstract}
We study the structures of Pfaffian equations
and contiguity relations of the hypergeometric function
of type $(k+1,k+n+2)$ 
by using twisted cohomology groups and the intersection form on them. 
We apply our results to algebraic statistics; 
numerical evaluation of the normalizing constants of two way contingency tables
with fixed marginal sums.
\end{abstract}

\section{Introduction}
We consider the hypergeometric integral of type $(k+1,k+n+2)$ 
defined as 
\begin{align*}
  F(\a;x)=\int_{\square}& \prod_{i=1}^k t_i^{\a_i} 
  \cdot \prod_{j=1}^n \Bigl( 1+\sum_{i=1}^kx_{ij}t_i \Bigr)^{\a_{k+j}} \\
  &\cdot \Bigl( 1+\sum_{i=1}^k t_i \Bigr)^{\a_{k+n+1}} 
  \cdot \frac{\dd t_1 \we \cdots \we \dd t_k}{t_1 \cdots t_k} ,
\end{align*}
where $\a_i$'s are parameters, $x_{ij}$'s are $k\times n$ variables, and 
$\square$ is a certain region. 
In this paper, we study Pfaffian equations 
and contiguity relations of the hypergeometric function $F(\a ;x)$
by using the twisted cohomology groups 
and the intersection forms.  
In \cite{M2} and \cite{G-FD}, a Pfaffian equation and contiguity relations 
of Lauricella's $F_D$ (the case of $k=1$) are studied 
in the same framework. 
This paper generalizes these results. 

We regard Pfaffian equations and contiguity relations
as matrix representations of some linear maps
on twisted cohomology groups. 
To obtain the matrices providing the relations, 
we use the intersection form of the twisted cohomology group
(Proposition \ref{prop:MJ-expression} and Theorem \ref{th:contiguity}). 
Our expressions have simple forms; 
each of the matrices for Pfaffian equations is determined by only an 
eigenvector with non-zero eigenvalue, 
and that for contiguity relations is decomposed into 
a product of intersection matrices and a diagonal one. 
An advantage of our method is that 
it systematically yields the relations 
from small initial data for a given basis of 
the twisted cohomology group without complicated calculations. 
Further, we give expressions of these linear maps by using 
the intersection forms (Theorems \ref{th:main} and 
\ref{th:contiguity-intersection}). 
They are independent of choice of bases of 
twisted cohomology groups. 

Finally, we discuss an application of our results
to algebraic statistics. 
We can express the normalizing constant of the hypergeometric distribution
of the $r_1 \times r_2$ two way contingency tables 
with fixed marginal sums by 
the hypergeometric function of type $(r_1,r_1+r_2)$ with integral parameters. 
By using contiguity relations, 
we give an algorithm (Algorithm \ref{algo:NC}) for
evaluating values of the normalizing constant 
in the framework of holonomic gradient method \cite{N3OST2}. 
Further, a Pfaffian equation gives the gradient matrix of 
the expectations, which is important  
to solve the conditional maximal likelihood estimate problem \cite[\S 4]{TKT}. 
We refer \cite{Ogawa-D} and \cite{TKT} for statistical applications
of hypergeometric functions. 


Pfaffian equations and contiguity relations have been studied 
from several points of view. 
In \cite{KM-duality},  
Kita and the second author give 
an expression of the Gauss-Manin connection for some basis 
of the twisted cohomology group. 
In fact, it does not directly imply a Pfaffian equation; 
see \S \ref{sec:pfaffian}.
Aomoto studies the contiguity relations of 
the hypergeometric functions of type $(k,n)$ by using twisted cohomology groups
in \cite{Aomoto}. 
This result is based on calculations in only the target space of 
$\Conti{i}$ in Proposition \ref{prop:linear-map}. 
On the other hand, 
by considering both of its domain and target spaces, 
we can clarify a structure of contiguity relations. 
Sasaki studies them  
in the framework of the Aomoto-Gel'fand system on the Grassmannian manifold in \cite{Sasaki}. 
However, it only gives contiguity relations on coordinates of 
the Grassmannian manifold in general case. 
Though Takayama gives an algorithmic method that uses Gr\"obner bases 
to derive the contiguity relations in \cite{Takayama}, 
this method requires huge computer resources. 
Recently in \cite{OT}, Ohara and Takayama give a numerical method 
to derive Pfaffian equations and contiguity relations of 
$A$-hypergeometric systems, one of which $\Fax$ satisfies, 
to evaluate the normalizing constant of $A$-hypergeometric distributions, 
but it is still difficult to get Pfaffian equations and contiguity relations 
unless $k$ and $n$ are small enough. 

For statistical applications, 
the method given in \cite{OT} is applicable to evaluation
not only for two way contingency tables
but also for other cases, 
while our algorithm is much faster 
and can solve larger problems
than theirs for the two way contingency tables. 


\section{Preliminaries}
\label{sec:preliminaries}
Let $Z=(z_{ij})$ be a square matrix arranged $n^2$ variables $z_{ij}$
$(1\le i,j\le n)$.
\begin{fact} 
\label{fact:dlog}
The logarithmic derivative of the determinant $|Z|$ of $Z$ is
$$\dd \log |Z|=\sum_{1\le i,j\le n}
\frac{(-1)^{i+j} |Z_{ij}|\dd z_{ij}}{|Z|},$$
where $Z_{ij}$ is the square matrix of size $n-1$ 
removing the $i$-th row and the $j$-th column from $Z$, 
and $|Z_{ij}|$ is its determinant.
\end{fact}
\begin{proof}
By the cofactor expansion 
$$|Z|=\sum_{p=1}^n (-1)^{p+i}z_{ip}|Z_{ip}|,$$
with respect to the $i$-th row of $Z$, 
we have 
$$\dd |Z|=\sum_{p=1}^n \Bigl( (-1)^{p+i} |Z_{ip}|\dd z_{ip}
+(-1)^{p+i}z_{ip}\dd |Z_{ip}| \Bigr) .
$$
Since the minor $|Z_{ip}|$ does not have the variable $z_{ij}$,  
the coefficient function of $\dd z_{ij}$ in $\dd |Z|$ is $(-1)^{i+j}|Z_{ij}|$.
This property together with the symmetry yields this fact.
\end{proof}

Let  $x_{ij}$ $(1\le i\le k,\ 1\le j\le n)$ be 
$k\times n$ variables and $x=(x_{ij})$ be the matrix arranging them.
We set a $(k+1)\times (k+n+2)$ matrix 
\begin{eqnarray*}
\tx &=&(\tx_{ij})_{\substack{0\le i\le k\hspace{8mm}\\ 0\le j\le k+n+1}}
=\begin{pmatrix}
1 & \mathbf{0}_k & \mathbf{1}_n & 1 \\
\tr\mathbf{0}_k& I_k & x &\tr \mathbf{1}_k
\end{pmatrix} \\
&=&
\bordermatrix{
 &0&1& \cdots &k&k+1& \cdots &k+n& k+n+1\cr
0&1&0& \cdots      &0&1  & \cdots &1 &1 \cr
1&0&1&   &0&x_{11}&\cdots&x_{1n} &1 \cr
\vdots&\vdots & &\ddots &\vdots &\vdots &\ddots&\vdots &\vdots \cr
k&0&0&\cdots  &1&x_{k1}&\cdots&x_{kn} &1 \cr
} ,
\end{eqnarray*}
where 
$\mathbf{0}_k=(0,\dots,0)\in \Z^k$, $\mathbf{1}_n=(1,\dots,1)\in \Z^n$, and
$I_k$ is the unit matrix of size $k$. 

Let $\cJ$ be the set of subsets of 
$\{0,1,2,\dots,k+n,k+n+1\}$ with cardinality $k+1$. 
Any element in $\cJ$ is expressed as 
$$J=\{j_0,j_1,\dots,j_k\}, \quad 
0\le j_0<j_1<\cdots <j_k\le k+n+1.$$
We set 
$$\tx \la J\ra=
\begin{pmatrix}
\tx_{0,j_0} &\tx_{0,j_1} &\cdots &\tx_{0,j_k} \\
\tx_{1,j_0} &\tx_{1,j_1} &\cdots &\tx_{1,j_k} \\
\vdots & \vdots &\ddots & \vdots \\
\tx_{k,j_0} &\tx_{k,j_1} &\cdots &\tx_{k,j_k}
\end{pmatrix},
$$
which is the sub-matrix of $\tx$ 
consisting of the $j_0$-th, $j_1$-th, $\dots$, $j_k$-th columns.
We define a subset of $\cJ$ by 
$$\cJ^\circ=\{J\in \cJ\mid \dx |\tx \la J\ra | \ne 0\},$$
where $\dx$ is the exterior derivative with respect to 
$x_{11}$, $x_{12}$, $\ldots$, $x_{kn}$.

\begin{lemma}
\label{lem:iJj+k}
Any element $J$ of $\cJ^\circ$ does not include an index $i$ 
with $1\le i\le k$ and includes an index $k+j$ with 
$1\le j\le n$.
\end{lemma}
\begin{proof}
If $J\in \cJ$ includes $1,\dots,k$ then $|\tx \la J\ra|=\pm1$. 
If $J$ includes none of $k+1,\dots,k+n$, then 
there is no variable $x_{ij}$ in $\tx \la J\ra$. 
\end{proof}

We count the cardinality of the set $\cJ^\circ$.
It is easy to that $\#(\cJ)=\binom{k+n+2}{k+1}$. 
If we choose $J=\{j_0,j_1,\dots,j_k\}\in \cJ$ so that 
$$j_0=1,\ j_1=2,\dots,\ j_{k-1}=k,$$
then $\tx\la J\ra$ becomes a constant for any $j_k$. There are $n+2$ ways 
to choose $j_k$. 
If we choose 
$$j_0=0,\ j_k=n+k+1,$$
then $\tx\la J\ra$ becomes a constant for 
$\{j_1,\dots,j_{k-1}\}\subset \{1,\dots,k\}$.  
There are $k$ ways to choose $j_1,\dots,j_{k-1}$.
Thus we have the following lemma.
\begin{lemma}
\label{lem:num-components}
The cardinality of $\cJ^\circ$ is 
$$
\binom{k+n+2}{k+1}-(k+n+2).
$$
\end{lemma}

Let $L_j$ $(0\le j\le k+n+1)$ be linear forms of $t_0,t_1,\dots,t_k$ defined by
$$(t_0,t_1,\dots,t_k)\tx=(L_0,L_1,\dots,L_k,L_{k+1},\dots,L_{k+n},L_{k+n+1}).
$$
Namely, 
\begin{align*}
  &L_j=t_j \ (0\le j\le k),\quad 
  L_{k+j}=t_0+t_1x_{1j}+\cdots+t_kx_{kj} \ (1\le j\le n), \\
  &L_{k+n+1}=t_0+t_1+\cdots+t_k .
\end{align*}
\begin{lemma}
\label{lem:ti-exp}
Let $J=\{j_0,\dots,j_k\}$ be an element of $\cJ$.
The linear form $t_i=L_i$ $(0\le i\le k)$ can be expressed in terms of 
$L_{j_0},\dots, L_{j_k}$ as 
$$
t_i=\frac{1}{|\tx\la J\ra|}\sum_{p=0}^k | \tx\la \under{j_p}{J}{i}\ra | L_{j_p},
$$
where 
$_{j_p}J_{i}=(J-\{j_p\})\cup \{i\}=\{j_0,\dots,j_{p-1},i,j_{p+1},\dots,j_k\}$. 
\end{lemma}
\begin{proof}
Since
$$(t_0,t_1,\dots,t_k)\tx\la J\ra =(L_{j_0},L_{j_1}\dots,L_{j_k}),$$
$t_i$'s are expressed as
$$(t_0,t_1,\dots,t_k) =(L_{j_0},L_{j_1}\dots,L_{j_k})\tx\la J\ra^{-1}.$$
We have only to write down the $i$-th column of the cofactor matrix of 
$\tx\la J\ra$.
\end{proof}

We regard $(t_0,t_1,\dots,t_k)$ as projective coordinates of $\P^k$ and 
$(t_1,\dots,t_k)$ as affine coordinates with setting $t_0=1$.
For  $J=\{ j_0,j_1,\dots,j_k \} \in \cJ$ 
we set 
$$\f\la J\ra =\dt\log(L_{j_1}/L_{j_0})\wedge \dt\log(L_{j_2}/L_{j_0})
\wedge \cdots \wedge\dt\log(L_{j_k}/L_{j_0}),$$
where $\dt$ is the exterior derivative with respect to $t_1,\dots,t_k$.

\begin{fact} 
We have
$$
\f\la J\ra=\frac{|\tx\la J\ra |\dd t}{\prod_{p=0}^k L_{j_p}}, \quad 
\dd t =\dd t_1\wedge \cdots \wedge\dd t_k .
$$
\end{fact}
\begin{proof}
We use the following identity:
\begin{equation}
\label{eq:log-form}
\f\la J\ra =\sum_{p=0}^k (-1)^p L_{j_p} 
\frac{\bigwedge_{0\le q\le p}^{q\ne p}
\dt L_{j_q}}{\prod_{q=0}^k L_{j_q}}.
\end{equation}
Consider the cofactor expansion of the $0$-th row of 
the sub-matrix $\tx \la J \ra$ of $\tx$.  
\end{proof}

\section{Gauss-Manin connections}
\label{section:cohomology}
Let $\a_0,\a_1,\dots,\a_{k+n},\a_{k+n+1}$ be parameters in $\C-\Z$ 
satisfying 
\begin{equation}
\label{eq:non-int}
\sum_{j=0}^{k+n+1} \a_j=0. 
\end{equation}
We set $\a=(\a_0,\a_1\dots,\a_{k+n},\a_{k+n+1})$. 
We often regard $\a_i$'s as indeterminants. 
For an element $f(\a)$ of the rational function field 
$\C (\a)=\C (\a_0 ,\ldots ,\a_{k+n+1})$, 
we put $f(\a)^{\vee}=f(-\a)$. 
For a matrix $A$ with entries in $\C(\a)$, 
we denote by $A^{\vee}$ the matrix 
operated ${}^{\vee}$ on each entry of $A$.

We define sets $X$ and $\X$ as 
\begin{eqnarray*}
X&=&\{ x\in M(k,n;\C )\mid  |\tx\la J\ra | \ne 0 \textrm{ for any } J
\},\\
\X&=&\{(t,x)\in \C^k\times X\mid 
\prod_{j=0}^{k+n+1} L_{j}\ne 0\}.
\end{eqnarray*}
We set $1$-forms $\w$ and $\w_x$ as 
\begin{eqnarray*}
\w&=&\sum_{j=1}^{k+n+1} \a_j \dt \log L_j\\
&=&\a_1\frac{\dd t_1}{t_1}+\cdots +\a_k\frac{\dd t_k}{t_k}+
\a_{k+1}\frac{x_{11}\dd t_1+\cdots+x_{k1}\dd t_k}
{1+t_1x_{11}+\cdots +t_kx_{k1}}+\cdots \\
& &+
\a_{k+n}\frac{x_{1n}\dd t_1+\cdots+x_{kn}\dd t_k}
{1+t_1x_{1n}+\cdots +t_kx_{kn}}+
\a_{k+n+1}\frac{\dd t_1+\cdots+\dd t_k}
{1+t_1+\cdots +t_k}
,
\end{eqnarray*}
\begin{eqnarray*}
\w_x&=&\sum_{j=1}^{n} \a_{k+j} \dx \log L_{k+j}
= \sum_{\substack{1\le i\le k\\ 1\le j\le n}}
\frac{\a_{k+j}t_i\dd x_{ij}}{L_{k+j}}\\
&=& \frac{\a_{k+1} t_1 \dd x_{11}}
{1+t_1x_{11}+\cdots +t_kx_{k1}}+\cdots +
\frac{\a_{k+1} t_k \dd x_{k1}}
{1+t_1x_{11}+\cdots +t_kx_{k1}}\\
& &+\cdots + 
\frac{\a_{k+j} t_i \dd x_{ij}}
{1+t_1x_{1j}+\cdots +t_kx_{kj}}
+\cdots  
\\
& &+\frac{\a_{k+n} t_1 \dd x_{1n}}
{1+t_1x_{1n}+\cdots +t_kx_{kn}}+\cdots +
\frac{\a_{k+n} t_k \dd x_{kn}}
{1+t_1x_{1n}+\cdots +t_kx_{kn}}.
\end{eqnarray*}
We define operators as 
$$\na^\a=\dt+\w\wedge, \quad \na_x^\a=\dx+\w_x\wedge,\quad 
\na_{ij}^\a=\frac{\pa}{\pa x_{ij}}+\frac{\a_{k+j}t_i}{L_{k+j}}.$$
Note that 
\begin{equation}
\label{eq:decomp-conn}
\na_x^\a
=\sum_{\substack{1\le i\le k\\ 1\le j\le n}} \dd x_{ij}\wedge \na_{ij}^\a.
\end{equation}

For a fixed $x \in X$, we have twisted cohomology groups 
\begin{eqnarray*}
H^k(\W^\bu(T_x),\na^\a)&=&\W^k(T_x)/\na^\a(\W^{k-1}(T_x)),\\
H^k(\W^\bu(T_x),\na^{-\a})&=&\W^k(T_x)/\na^{-\a}(\W^{k-1}(T_x)) ,
\end{eqnarray*}
where $T_x$ is the preimage of $x$ under the projection 
$\X \ni (t,x) \mapsto x \in X$, and 
$\W^l (T_x)$ is the vector space of rational $l$-forms on $\P^k$ 
with poles only along $\P^k -T_x$. 
Here, we identify $T_x$ with an open subset of $\P^k$.

\begin{fact}[\cite{AK}]
The twisted cohomology groups  $H^k(\W^\bu(T_x),\na^{\pm\a})$ are
of rank $\binom{k+n}{k}$.
\end{fact}

There is the intersection pairing $\cI$ between 
$H^k(\W^\bu(T_x),\na^{\a})$ and $H^k(\W^\bu(T_x),\na^{-\a})$.

\begin{fact}[\cite{M1}]
\label{fact:intersection}
For $J=\{ j_0,\dots,j_k \}$ and $J'=\{ j'_0,\dots,j'_k \}$, we have  
$$
\cI(\f\la J\ra,\f\la J'\ra)=
\left\{
\begin{array}{cl}
(2\pi\sqrt{-1})^k \cdot 
\dfrac{\sum_{j\in J}\a_{j}}
{\prod_{j\in J}\a_{j}}
&\textrm{if } J=J',\\[5mm]
(2\pi\sqrt{-1})^k \cdot 
\dfrac{(-1)^{p+q}}
{\prod_{j\in J\cap J'}\a_{j}}
&\textrm{if } \#(J\cap J')=k,\\[5mm]
0&\textrm{otherwise,}
\end{array}
\right.
$$
where we assume that $J-\{j_p\}=J'-\{j'_q\}$ 
in the case of $\#(J\cap J')=k$.
\end{fact}
Note that if we regard $\a_i$'s as indeterminants, we have 
$\cI(\f\la J\ra,\f\la J'\ra)^{\vee}=(-1)^k \cdot \cI(\f\la J\ra,\f\la J'\ra)$.

\begin{proposition}
\label{prop:basis}
Let $p$ and $q$ be different two elements of the set 
$\{0,1,\dots,k+n+1\}$. We set 
\begin{eqnarray*}
_{q}\cJ_{p}&=&\{J\in \cJ\mid q\notin J,\ p\in J\},\\
_{p}\cJ_{q}&=&\{J\in \cJ\mid p\notin J,\ q\in J\}.
\end{eqnarray*}
Then $\{\f\la J\ra\}_{J\in \under{p}{\cJ}{q}}$ and 
$\{\f\la J'\ra\}_{J'\in \under{q}{\cJ}{p}}$ 
are bases of 
$H^k(\W^\bu(T_x),\na^{\pm\a})$.
\end{proposition}
\begin{proof}
Set $J=\{p,j_1,\dots,j_k\}\in \under{q}{\cJ}{p}$ and 
$J'=\{j_1,\dots,j_k,q\}\in \under{p}{\cJ}{q}$.
Note that there are $\binom{k+n}{k}$ ways to get such $J$ 
and $J'$ for fixed $p$ and $q$.
We align $\f\la J\ra$'s and $\f\la J'\ra$'s
by the lexicographic order of $(j_1,\dots,j_k)$.
Then the intersection matrix for $\f \la J \ra$ and $\f \la J' \ra$ becomes 
diagonal matrix with diagonal entries 
$$\frac{(-2\pi\sqrt{-1})^k}{\a_{j_1}\cdots\a_{j_k}}.$$
Thus they are bases of 
$H^k(\W^\bu(T_x),\na^{\pm\a})$.
\end{proof}

We define vector bundles of rank $r=\binom{k+n}{k}$ over $X$ with fibers 
$H^k(\W^\bu(T_x),\na^{\a})$ and $H^k(\W^\bu(T_x),\na^{-\a})$ by
$$\cH^\a=\bigcup_{x\in X} H^k(\W^\bu(T_x),\na^{\a}),\quad  
\cH^{-\a}=\bigcup_{x\in X} H^k(\W^\bu(T_x),\na^{-\a}),$$
respectively. We can regard 
$\{\f\la J\ra\}_{J\in \under{p}{\cJ}{q}}$ 
and $\{\f\la J'\ra\}_{J'\in \under{q}{\cJ}{p}}$ 
as global frames of these vector bundles. 
The operators $\na_x^{\a}$ and $\na_x^{-\a}$ are regarded as connections 
on $\cH^\a$ and $\cH^{-\a}$, i.e., they are $\C(\a)$-linear maps
$$\na_x^{\pm\a}:\Gamma(\cH^{\pm\a})\to \Gamma(\W^1(X)\otimes \cH^{\pm\a})$$
satisfying 
$$
\na_x^{\pm\a}(f\f)=d_x f\otimes \f+f\na_x^{\pm\a}(\f), 
$$
for $f\in \W^0(X)$ and $\f\in \Gamma(\cH^{\pm\a})$, 
where $\W^l(X)$ is a space of rational $l$-forms 
with poles only along the complement of $X$, 
$\Gamma(\cV)$ denotes the $\W^0(X)$-module of sections of $\cV$.
They are called the Gauss-Manin connections on 
the vector bundles $\cH^{\pm \a}$.
Note that
\begin{align}
\label{eq:pd-GMpd}
\displaystyle{\frac{\pa}{\pa x_{ij}} \int_\square
\Big(\prod_{j=1}^{n+k+1} L_j^{\a_j}\Big) \f}
&\displaystyle{=\int_\square\Big(\prod_{j=1}^{n+k+1} L_j^{\a_j}\Big) 
\naa_{ij}\big(\f\big),} \\
\label{eq:exd-GM}
\displaystyle{d_x \int_\square\Big(\prod_{j=1}^{n+k+1} L_j^{\a_j}\Big) \f}
&\displaystyle{=\int_\square\Big(\prod_{j=1}^{n+k+1} L_j^{\a_j}\Big) 
\naa_{x}\big(\f\big), }
\end{align}
for $\f\in \Gamma(\cH^\a)$, where $\square$ is a twisted cycle 
associated with $\prod_{j=1}^{n+k+1} L_j^{\a_j}$ 
(refer to \cite[\S 3.2]{AK} for its definition). 
These mean that the partial differential operators $\dfrac{\pa}{\pa x_{ij}}$ 
and the exterior derivative $\dx$ are translated into 
the operators $\naa_{ij}$ and the connection $\na_x^\a$ through 
the integration with respect to 
the kernel function $\prod_{j=1}^{n+k+1} L_j^{\a_j}$.

The intersection form $\cI$ is extended to a pairing 
between $\Gamma(\cH^\a)$ and $\Gamma(\cH^{-\a})$. 
We can regard $\cH^{-\a}$ as the dual of $\cH^{\a}$ by 
the intersection form $\cI$, since $\cI$ is a perfect pairing. 
By the compatibility of the connections and the intersection form, 
we have the following.

\begin{proposition}
\label{prop:parallel}
The intersection form $\cI$ satisfies 
$$\dx\cI(\f,\f')=\cI(\na_x^\a(\f),\f')+\cI(\f,\na_x^{-\a}(\f')),$$
for $\f\in \Gamma(\cH^\a)$ and $\f'\in \Gamma(\cH^{-\a})$.
\end{proposition}

Let $\Gamma_0(\cH^\a)$ (resp. $\Gamma_0(\cH^{-\a})$)  
be the vector space in $\Gamma(\cH^\a)$ (resp. $\Gamma(\cH^{-\a})$) 
spanned by $\f\la J\ra$'s $(J\in \cJ)$ over the field $\C(\a)$. 
Then any elements $\f\in \Gamma_0(\cH^\a)$ and $\f'\in \Gamma_0(\cH^{-\a})$ 
satisfy 
\begin{equation}
\label{eq:flat}
\dx\cI(\f,\f')=0
\end{equation} 
by Fact \ref{fact:intersection}.


We put $\dot{\cJ}=\under{k+n+1}{\cJ}{0}$, and align its elements lexicographically. 
We denote $\dot{\cJ} =\{ J^1 ,J^2,\ldots ,J^r \}$, $r=\binom{k+n}{k}$, and  
define a column vector $\Phi =\tr ( \f\la J^1\ra , \f\la J^2 \ra ,\ldots ,\f\la J^r \ra )$. 
We compute the connection matrix $\Psiax$ of $\na_x^\a$  
with respect to the frame 
$\{ \f\la J\ra \}_{J\in \dot{\cJ}}$; 
i.e., it satisfies 
\begin{equation}
\label{eq:GM-connection}
\na_x^\a \Phi =\Psiax \cdot \Phi. 
\end{equation}
We remark that 
$\Psiax$ is a square matrix of size $\binom{k+n}{k}$ with entries 
in $\W^1(X)$. 
We also express $\na_x^\a$  by the intersection form $\cI$ without 
taking a frame.
By the decomposition (\ref{eq:decomp-conn}) of $\na_x^\a$, we have 
$$\Psiax=\sum_{\substack{1\le i\le k\\1\le j\le n}}
\Psiij \dd x_{ij},$$
where the matrix $\Psiij$ is obtained by 
the action of operator $\naa_{ij}$ on the frame $\{ \f \la J\ra\}_{J\in \dot{\cJ}}$.
We study the operator $\na_{ij}^\a$.
\begin{lemma}
\label{lem:na-ij-simple}
Let $J=\{j_0,\dots,j_k\}$ be an element of $\cJ$.
Suppose that 
$$1\le i\le k,\quad1\le j\le n,\quad  k+j\notin J.$$
Then 
$$\naa_{ij} (\f\la J\ra )=
\a_{k+j}\sum_{p=0}^k 
\frac{|\tx\la \under{j_p}{J}{i}\ra|}{|\tx\la \under{j_p}{J}{k+j}\ra|}
\f\la \under{j_p}{J}{k+j} \ra,$$
where 
\begin{eqnarray*}
_{j_p}J_i&=&\{j_0,\dots,j_{p-1},i,j_{p+1},\dots,j_k\},\\
_{j_p}J_{k+j}&=&\{j_0,\dots,j_{p-1},k+j,j_{p+1},\dots,j_k\}.
\end{eqnarray*}
\end{lemma}
\begin{proof}
Since $(\pa/\pa x_{ij})\f\la J\ra=0$, we have 
\begin{eqnarray*}
& &\naa_{ij} (\f\la J\ra )=\frac{\a_{k+j}t_i}{L_{k+j}}\f\la J\ra\\
&=&\frac{\a_{k+j}}{|\tx\la J\ra|}
\left(\sum_{p=0}^k|\tx\la \under{j_p}{J}{i}\ra|  L_{j_p}\right)
\frac{|\tx\la J\ra|\dd t}{L_{k+j}L_{j_0}L_{j_1}\cdots L_{j_k}}\\
&=&
\a_{k+j}\sum_{p=0}^k 
\frac{|\tx\la \under{j_p}{J}{i}\ra|}{|\tx\la \under{j_p}{J}{k+j}\ra|}
\f\la \under{j_p}{J}{k+j} \ra .
\end{eqnarray*}
Here we use Lemma \ref{lem:ti-exp} for the expression of $t_i$.
\end{proof}

For a fixed $i$ and $j$, we take a basis $\f\la J\ra$'s
of $H^k(\W^\bu(T_x),\na^{\a})$, where  
$J$ runs over the subset 
$_{k+j}\cJ_i=\{J\in \cJ\mid k+j\notin J,\ i\in J\}$
of $\cJ$
with cardinality $\binom{k+n}{k}$.  
Under this condition, we have 
$$ \frac{|\tx\la \under{j_p}{J}{i}\ra|}{|\tx\la \under{j_p}{J}{k+j}\ra|}\dd x_{ij}
=(-1)^{p+i}\frac{\pa}{\pa x_{ij}}\log | \tx\la \under{j_p}{J}{k+j}\ra | \dd x_{ij},
$$
by Fact \ref{fact:dlog} and 
$$
\tx\la \under{j_p}{J}{i}\ra=
\bordermatrix{
 &0&     \cdots &p-1  &p& p+1& \cdots &k\cr
0&\tx_{0,j_0}&\cdots&\tx_{0,j_{p-1}}&0&\tx_{0,j_{p+1}}
 &\cdots&\tx_{0,j_{k}}\cr
\vdots& \vdots&\cdots&\vdots &\vdots &\vdots&\cdots&\vdots\cr
i&\tx_{i,j_0}&\cdots&\tx_{i,j_{p-1}}&1&\tx_{i,j_{p+1}}
 &\cdots&\tx_{i,j_{k}}\cr
\vdots& \vdots&\cdots&\vdots &\vdots &\vdots&\cdots&\vdots\cr
k&\tx_{k,j_0}&\cdots&\tx_{k,j_{p-1}}&0&\tx_{k,j_{p+1}}
 &\cdots&\tx_{k,j_{k}}\cr
},
$$
$$
\tx\la \under{j_p}{J}{k+j}\ra=
\bordermatrix{
 &0&     \cdots &p-1  &p& p+1& \cdots &k\cr
0&\tx_{0,j_0}&\cdots&\tx_{0,j_{p-1}}&1&\tx_{0,j_{p+1}} &\cdots&\tx_{0,j_{k}}\cr
\vdots& \vdots&\cdots&\vdots &\vdots &\vdots&\cdots&\vdots\cr
i&\tx_{i,j_0}&\cdots&\tx_{i,j_{p-1}}&x_{i,j}&\tx_{i,j_{p+1}}
 &\cdots&\tx_{i,j_{k}}\cr
\vdots& \vdots&\cdots&\vdots &\vdots &\vdots&\cdots&\vdots\cr
k&\tx_{k,j_0}&\cdots&\tx_{k,j_{p-1}}&x_{k,j}&\tx_{k,j_{p+1}}
 &\cdots&\tx_{k,j_{k}}\cr
}.
$$
Note that 
the basis change transformation matrix 
from $\{\f \la J \ra \}_{J\in \under{k+j}{\cJ}{i}}$ to $\{\f \la J \ra \}_{J\in \dot\cJ}$ 
is independent of $x_{ij}$.  
Thus the coefficient $\Psiij$ of $\dd x_{ij}$ in the connection matrix 
$\Psiax$ can be expressed as 
\begin{align}
  \label{eq:psi-ij}
  \Psiij =\sum_{J\in \cJ^\circ, J\ni k+j} 
\MJij \frac{\pa}{\pa x_{ij}} \log |\tx\la J\ra |,
\end{align}
where $\MJij$ are square matrices of size 
$r=\binom{k+n}{k}$ which are independent of 
any entries of $x$.

\begin{lemma}
\label{lem:eigenvals}
Suppose that $\a_J=\a_{j_0}+\cdots +\a_{j_k}\ne 0$ for 
$J=\{j_0,\dots,j_k\}\in \cJ^\circ$ 
with $i\notin J$, $k+j\in J$.
\begin{enumerate}
\item 
The matrices $\MJij$ are independent of $i$ and $j$. 
\item 
The eigenvalues of the matrix $\MJij$ are $\a_J$ and $0$. 
\item 
Its eigenspace of eigenvalue $\a_J$ is $1$-dimensional, and 
that of eigenvalue $0$ is $(r-1)$-dimensional. 
\end{enumerate}
\end{lemma}

\begin{proof}
We may assume $j_0=k+j$. 
We set 
$$\under{j_0}{\cJ''}{i}=(\{J\}\cup \under{j_0}{\cJ}{i}) -
\{\under{j_0}{J}{i}\},$$
and assume that its first entry is $J$.

We take a point $\dot x\in X$ so that 
$L_j$ ($j\in J$) form a real small simplex.
Let $U_{\dot x}$ be a small open set in $\C^{k\times n}$ 
including $\dot x$ and points with $|\tx\la J\ra |=0$.
For any $J''\in \under{j_0}{\cJ''}{i}$, we make a twisted cycle 
$\De\la J''\ra$ by using $L_j$ $(j\in J'')$. 
By computing the intersection numbers of 
$\De\la J''\ra$,  we can show that they form a basis of 
a twisted homology group. 
We construct a period matrix 
$$\Pi^\a(x)=\left(\int_{\De\la J''\ra} 
\Big(\prod_{j=1}^{k+n+1} L_j^{\a_j}\Big) \f\la J'\ra
\right)_{J'\in \dot \cJ, J''\in \under{j_0}{ \cJ''}{i}}
$$
on $U_{\dot x}\cap X$. By the perfectness of the pairing between 
the twisted homology and cohomology groups, it is invertible. 
When $x$ turns around the divisor $|\tx\la J\ra |=0$, 
the argument of each $L_i$ on $\De\la J''\ra$ 
is almost unchanged for $J''\in \under{j_0}{ \cJ''}{i}-\{J\}$. 
Moreover, the integrals over $\De\la J''\ra$ $(J''\in \under{j_0}{\cJ ''}{i}-\{J\})$
are valid on the divisor $|\tilde x\la J\ra|=0$ in $U_{\dot x}$. 
Thus the entries $\Pi^\a(x)$ except in the first column 
are single-valued and holomorphic on $U_{\dot x}$. 
We consider the behavior of the first column of $\Pi^\a(x)$.
To compute the integrals, we use the coordinate change such that 
$L_j$ ($j\in J$) are expressed as $t_1,\dots,t_k,$ $1+t_1+\cdots+t_k$.
This coordinate change is equivalent to the left multiplication ${\tx}'$ of 
$\tx\la 1,\dots k,k+n+1\ra\tx\la J\ra^{-1}$ to $\tx$. 
Let $L_j'$ be linear forms corresponding to 
the matrix ${\tx}'$. We have 
$$
\int_{\De\la J\ra} 
\Big(\prod_{j=1}^{k+n+1} L_j^{\a_j}\Big) \f\la J'\ra
=\int_{\De'} 
\Big(\prod_{j=1}^{k+n+1} {L'_j}^{\a_j}\Big) \f'\la J'\ra,
$$
where $\f'\la J'\ra$ is naturally defined by ${\tx}'$ and $\De'$ is the 
regularization of a standard simplex 
\begin{equation}
\label{eq:s-simplex}
\De
=\{(t_1,\dots,t_k)\in \R^k\mid t_1,\dots,t_k<0,\; t_1+\cdots+t_k>-1\} 
\end{equation}
with respect to $\prod_{j=1}^{k+n+1} {L'_j}^{\a_j}$.
Here note that every linear form $L_{q}'$ ($q\notin J$) 
has the factor $|\tx\la J\ra|^{-1}$. By taking out this factor from 
this integral, we see that each entry in the first column of $\Pi^\a(x)$ is 
the product of 
$$
\prod_{q\notin J} |\tx\la J\ra|^{-\a_{q}}
=|\tx\la J\ra|^{\a_{J}}
$$ 
and a single-valued holomorphic function on $U_{\dot x}$.
Here note that  we use the assumption (\ref{eq:non-int}).
Hence we have a local expression 
$$\Pi^\a(x)=\Pi^\a_J(x)D^\a_J(x),\quad 
D^\a_J(x)=\diag(|\tx\la J\ra|^{\a_J},1\dots,1)
$$
around $\dot x$, where $\Pi^\a_J(x)$ is 
a single-valued holomorphic matrix function on $U_{\dot x}$, 
and $\diag(c_1,\ldots ,c_r)$ denotes the diagonal matrix 
with diagonal entries $c_1,\ldots ,c_r$.
By operating $\dx$ on the both sides of the above and using 
the equalities (\ref{eq:exd-GM}) and (\ref{eq:GM-connection}), we have 
\begin{align*}
&\dx\Pi^\a(x)=\Psi(\a;x) \Pi^\a(x) 
=\dx \Pi^\a_J(x)D^\a_J(x)+\Pi^\a_J(x)\dx D^\a_J(x)\\
=&(\dx \Pi^\a_J(x)D^\a_J(x)+\Pi^\a_J(x)\dx D^\a_J(x))
\cdot (\Pi^\a_J(x)D^\a_J(x))^{-1}\cdot \Pi^\a(x)\\
=&\Big[\dx \Pi^\a_J(x)\Pi^\a_J(x)^{-1}+\Pi^\a_J(x)\diag(\a_J,0,\dots,0)
\Pi^\a_J(x)^{-1}\frac{\dx|\tx\la J\ra|}{|\tx\la J\ra|}
\Big]\Pi^\a(x).
\end{align*}
Since 
$$\Psi(\a;x)=\dx \Pi^\a_J(x)\Pi^\a_J(x)^{-1}+\Pi^\a_J(x)\diag(\a_J,0,\dots,0)
\Pi^\a_J(x)^{-1}\frac{\dx|\tx\la J\ra|}{|\tx\la J\ra|},$$
we have this lemma.
\end{proof}

Hereafter, the matrix $\MJij$ is denoted by simply $\MJa$. 
Let $\cMJ$ be the linear transformation of 
$\Gamma_0(\cH^\a)$ corresponding to the matrix $\MJa$.

\begin{lemma}
\label{lem:0-eigensp}
Suppose that the same assumption as Lemma \ref{lem:eigenvals}.
\begin{enumerate}
\item 
Let $\f$ and $\f'$ be an element of the eigenspace of $\cMJ$ 
of eigenvalue $\a_J$ and that of eigenvalue $0$, respectively.
Then 
$\f'$ represents an element of the eigenspace of $\cM_J(-\a)$ of eigenvalue $0$, 
and $\f$ and $\f'$ satisfy 
$$\cI(\f,\f')=0.$$
\item 
The eigenspace of $\cMJ$ of eigenvalue $0$ 
is spanned by $\f\la J'\ra$ for $J'\in \under{k+j}{\cJ}{i}
-\{\under{k+j}{J}{i}\}$.   
\item 
The eigenspace of $\cMJ$ of eigenvalue $\a_J$  
is spanned by $\f\la J\ra$.   
\end{enumerate}
\end{lemma}
\begin{proof}

\noindent(1) 
By replacing $\a$ to $-\a$ for $\cM_J(\a)\f'=0$, we have $\cM_J(-\a)\f'=0$. 
Then $\f'$ represents a $0$-eigenvector of $\cM_J(-\a)$. 
Proposition \ref{prop:parallel} together with (\ref{eq:flat}) implies that  
$$0=\dx\cI(\f,\f')=\cI(\na_x^\a \f,\f')+\cI(\f,\na_x^{-\a}\f').$$
Thus we have 
$$0=\cI(\naa_{ij} \f,\f')+\cI(\f,\na_{ij}^{-\a}\f')
=\a_J\cdot  \cI( \f,\f')+0\cdot \cI(\f,\f').
$$
Since $\a_J\ne 0$,  $\cI( \f,\f')$ should be $0$.

\noindent(2) 
The matrix $\MJa=\MJij$ is the coefficient of 
$\frac{\pa}{\pa x_{ij}}\log |\tx\la J\ra | \dd x_{ij}$ 
by the action 
of $\dd x_{ij}\wedge \naa_{ij}$. 
Let $J'=\{ j_0'=i,j_1',\dots,j_k'\}$ be an element of $_{k+j}\cJ_{i}$. 
It satisfies the assumption of Lemma \ref{lem:na-ij-simple}.
We consider the condition the factor $|\tx\la J\ra|$ appears in 
the denominator of 
$$\naa_{ij} (\f\la J'\ra )=
\a_{k+j}\sum_{p=0}^k 
\frac{|\tx \la \under{j'_p}{J'}{i}\ra|}{|\tx \la \under{j'_p}{J'}{k+j}\ra|}
\f\la \under{j'_p}{J'}{k+j} \ra.$$
This case only happens 
$$p=0;\quad j'_p (=j'_0)=i,\ j'_1=j_1,\ \dots,\ j'_k=j_k.$$
Otherwise, the factor $|\tx\la J\ra|$ never appears 
$\naa_{ij} (\f\la J'\ra )$ for $J'\in \under{k+j}{\cJ}{i}$.

\noindent(3) 
By Fact \ref{fact:intersection}, we have 
$$\cI(\f\la J\ra, \f\la J'\ra)=0$$
for  any $J'\in \under{k+j}{\cJ}{i}
-\{\under{k+j}{J}{i}\}$. Lemma \ref{lem:eigenvals} together with 
(1) implies the claim.
\end{proof}

Recall that the index set 
\begin{align*}
  \dot\cJ=\under{k+n+1}{\cJ}{0}&=\{J=\{0,j_1,\dots,j_k\}\mid 1\le j_1<\cdots <j_k\le k+n\} \\
  &=\{ J^1=\dot{J} ,J^2,\ldots ,J^r \}\quad \Big (r=\binom{k+n}{k} \Big)
\end{align*}
of the basis $\{\f\la J\ra \}_{J\in \dot{\cJ}}$ is aligned lexicographically, and  
the column vector $\Phi$
is defined as $\Phi=\tr ( \f\la J^1\ra , \ldots ,\f\la J^r \ra )$. 
Let $C(\a)$ be the intersection matrix of this basis. 
For any $J\in \cJ$, 
let $v_J$ be the row vector defined as 
$$
v_J=\Bigl( \cI(\f\la J\ra, \f\la J^1 \ra),\dots ,\cI(\f\la J\ra, \f\la J^r \ra)\Bigr)
\cdot C(\a)^{-1} .
$$
Then $\f\la J\ra$ is expressed as 
$\f\la J\ra=v_J\Phi$, and we have 
$v_J C(\a) \tr v_{J'}^{\vee} =\cI(\f\la J\ra,\f\la J'\ra)$. 
\begin{lemma} Suppose that $\a_J\ne 0$.
  The row vector $v_J$ is a row eigenvector of $\MJij$ 
  with eigenvalue $\a_J$.   
  An arbitrary $0$-eigenvector $v$ of $\MJij$ satisfies 
  $v C(\a) \tr v_J^{\vee}=0$. 
\end{lemma}
\begin{proof}
  The lemma follows from Lemma \ref{lem:0-eigensp}
and 
  Fact \ref{fact:intersection}. 
\end{proof}
By this lemma, we obtain the following. 
\begin{proposition}
  \label{prop:MJ-expression}
  The matrix $\MJa$ is expressed as 
  \begin{align}
    \label{eq:psi-expression}
\a_J
C(\a)\tr v_J^{\vee} v_J(v_J C(\a) \tr v_J^\vee)^{-1}=
    \frac{\prod_{p\in J}\a_p}{(2\pi\sqrt{-1})^k}
    C(\a)\tr v_J^{\vee} v_J.
  \end{align}
\end{proposition}
\begin{proof}
We suppose temporarily $\a_J\ne0$. Lemma \ref{lem:0-eigensp} yields that 
the row vector $v_J$ is a row eigenvector of $\MJij$ 
with eigenvalue $\a_J$, and that    
an arbitrary $0$-eigenvector $v$ of $\MJij$ satisfies 
 $v C(\a) \tr v_J^{\vee}=0$.

It is easy to see that the eigenspaces of $\MJa$ coincide with 
those of the left hand side of (\ref{eq:psi-expression}).
Since the factor $\a_J$ is canceled with 
$v_J C(\a) \tr v_J^\vee$, we have the identity (\ref{eq:psi-expression}),  
which is valid even in the case $\a_J=0$. 
\end{proof}
\begin{remark}
  If we write $C(\a)$ as the form $(2\pi\sqrt{-1})^k C'(\a)$, then we can cancel
  the factor $(2\pi\sqrt{-1})^k$ in the right-hand side of (\ref{eq:psi-expression}). 
\end{remark}
\begin{theorem}
  \label{th:Psi-expression}
  The connection matrix $\Psiax$ is expressed as 
  \begin{align}
    \label{eq:pfaffian-expression}
    \Psiax=\sum_{J\in \cJ^\circ} \MJa \dx \log |\tx\la J\ra | ,
  \end{align}
where the explicit form of $\MJa$ is given as 
(\ref{eq:psi-expression}) in Proposition \ref{prop:MJ-expression}.
\end{theorem}
\begin{proof}
  We have
  \begin{align*}
    \Psiax
    &=\sum_{J\in \cJ^\circ} \Bigl( \MJa 
    \sum_{i=1}^k \sum_{j}^{k+j\in J} \frac{\pa}{\pa x_{ij}} \log |\tx\la J\ra | \dd x_{ij} \Bigr) \\
    &=\sum_{J\in \cJ^\circ} \MJa \dx \log |\tx\la J\ra | , 
  \end{align*}
  by (\ref{eq:psi-ij}). 
\end{proof}
We can express the connection $\naa_x$ by the intersection form. 
This expression is independent of choice of 
a frame of $\Gamma_0(\cH^\a)$.
\begin{theorem}
\label{th:main} For any element $\f\in \Gamma_0(\cH^\a)$, we have 
\begin{eqnarray*}
\naa_x(\f)&=&
\sum_{J\in \cJ^\circ} 
\a_J\frac{\cI(\f,\f\la J\ra)}{\cI(\f\la J\ra,\f\la J\ra)}\f \la J \ra
\dx \log |\tx\la J\ra | \\
&=&\frac{1}{(2\pi\sqrt{-1})^k}
\sum_{J\in \cJ^\circ} 
\Big(\prod_{j\in J}\a_{j}\Big)\cI(\f,\f\la J\ra)\f \la J \ra
\dx \log |\tx\la J\ra |.
\end{eqnarray*}
\end{theorem}
\begin{proof}
  Theorem \ref{th:Psi-expression} implies that $\naa_x$ can be 
  expressed as a linear combination of $\dx \log |\tx\la J\ra |$. 
We consider the linear transformation 
$$
\Gamma_0(\cH^\a)\ni 
\f\mapsto
\a_J\frac{\cI(\f,\f\la J\ra)}{\cI(\f\la J\ra,\f\la J\ra)}\f \la J \ra
\in \Gamma_0(\cH^\a).
$$
By comparing the eigenspaces of this transformation 
with those of $\cMJ$ given in Lemma \ref{lem:0-eigensp}, 
we conclude that it coincides with $\cMJ$ under the condition $\a_J\ne0$. 
Note that the factor $\a_J$ is canceled with $\cI(\f\la J\ra,\f\la J\ra)$. 
Thus the connection $\naa_x$ admits the expressions.
\end{proof}

\section{Pfaffian equations}
\label{sec:pfaffian}
A Pfaffian equation of $F(x)$ means a first order 
linear differential equation for a vector-valued unknown function including 
$F(x)$ which is integrable and equivalent to a holonomic system of 
linear differential equations annihilating the single unknown function $F(x)$.

Via the equality (\ref{eq:exd-GM}), 
we can regard the equation (\ref{eq:GM-connection}) 
as a first order differential equation for a vector-valued 
function of rank $r=\binom{k+n}{k}$. It satisfies the 
integrability condition  
$$\dx \Psi(\a;x)=\Psi(\a;x)\wedge \Psi(\a;x).$$
Thus for any point $x\in X$, 
there exists a unique solution to this differential equation 
around $x$ under an initial condition. 
However, we cannot immediately regard it as a Pfaffian equation of 
$$
\Fax=\int_\square\prod_{j=1}^{n+k+1} L_j^{\a_j} \f\la \dot J\ra,  
$$
where $\dot J=\{0,1,\dots,k\}$ and $\square$ is a twisted cycle. 
To obtain differential equations annihilating $\Fax$ from it, 
we need to express 
$$\int_\square\prod_{j=1}^{n+k+1} L_j^{\a_j} \f\la J\ra \quad (J\in 
\under{k+n+1}{\cJ}{0})
$$
by actions of 
the ring of differential operators with rational function coefficients
$\C(\dots,x_{ij},\dots)\la \dots,\dfrac{\pa}{\pa x_{ij}},\dots\ra$ 
on $\Fax$. 
In this section, by differentiating $\Fax$ several times, 
we find a vector-valued function $\bFax$ such that 
it satisfies a Pfaffian equation with the connection matrix $\Psi(\a,;x)$.

By the equality of (\ref{eq:pd-GMpd}), we can translate computations of  
$\dfrac{\pa}{\pa x_{ij}}\Fax$ to those of $\naa_{ij}(\f\la \dot J\ra)$. 
Firstly, we express $\naa_{ij}(\f\la \dot J\ra)$ in terms of $\f\la J\ra$.
Since $(\pa/\pa x_{ij})( \f\la \dot J\ra)=0$, 
we have 
$$\naa_{ij}(\f\la \dot J\ra)=
\frac{\a_{k+j}\dd t}{L_{k+j}t_1\cdots t_{i-1}t_{i+1}\cdots t_{k}}
=\frac{\a_{k+j}}{|\tx \la \under{i}{\dot J}{k+j}\ra|}\f\la \under{i}{\dot J}{k+j}\ra,
$$
where 
$$_i\dot J_{k+j}=(\dot J-\{i\})\cup \{k+j\}=
\{0,1,\dots,i-1,k+j,i+1,\dots,k\}.$$
By these operators, we obtain $k\times n$ functions.

Secondly, we express $(\naa_{i'j'}\circ \naa_{ij})(\f\la \dot J\ra)$ 
in terms of 
$\f\la J\ra$, 
where $i\ne i'$ and $j\ne j'$.  
We have 
\begin{eqnarray*}
& &(\naa_{i'j'}\circ \naa_{ij})(\f\la \dot J\ra)\\
&=&\naa_{i'j'}\left(
\frac{\a_{k+j}\dd t}{L_{k+j}t_1\cdots t_{i-1}t_{i+1}\cdots t_{k}}\right)\\
&=&\frac{\a_{k+j'}\a_{k+j}\dd t}{L_{k+j'}L_{k+j}
t_1\cdots t_{i-1}t_{i+1}\cdots t_{i'-1}t_{i'+1}\cdots t_{k}}\\
&=&\frac{\a_{k+j}\a_{k+j'}}{|\tx\la \under{i',i}{\dot J}{k+j,k+j'}\ra|}
\f\la\under{i',i}{\dot J}{k+j,k+j'}\ra,
\end{eqnarray*}
where $_{i',i}\dot J_{k+j,k+j'}=(\dot J-\{i,i'\})\cup\{k+j,k+j'\}$. 
Since 
$$\naa_{i'j}\circ \naa_{ij'}=\naa_{ij'}\circ \naa_{i'j}=\naa_{ij}\circ \naa_{i'j'}
=\naa_{i'j'}\circ \naa_{ij},$$
we obtain 
$
\binom{k}{2}\times \binom{n}{2}
$
functions, which is the number of the way to choose $i_1,i_2,j_1,j_2$ such that 
$1\le i_1<i_2\le k,\ 1\le j_1<j_2\le n$. 

Thirdly, we act $\naa_{i_3j_3}\circ\naa_{i_2j_2}\circ\naa_{i_1j_1}$ on 
$\f\la \dot J\ra$, where 
$$\{i_1,i_2,i_3\}\subset\{1,\dots,k\},\quad 
\{j_1,j_2,j_3\}\subset\{1,\dots,n\},
$$
and their cardinalities are $3$.
We have 
\begin{eqnarray*}
& &(\naa_{i_3j_3}\circ \naa_{i_2j_2}\circ \naa_{i_1j_1})(\f\la \dot J\ra)\\
&=&\frac{\a_{k+j_1}\a_{k+j_2}\a_{k+j_3}t_{i_1}t_{i_2}t_{i_3}\dd t}
{L_{k+j_1}L_{k+j_2}L_{k+j_3}
t_1\cdots t_{k}}\\
&=&\frac{\a_{k+j_1}\a_{k+j_2}\a_{k+j_3}}
{|\tx\la \under{i_3,i_2,i_1}{\dot J}{k+j_1,k+j_2,k+j_3}\ra|}
\f\la \under{i_3,i_2,i_1}{\dot J}{k+j_1,k+j_2,k+j_3}\ra,
\end{eqnarray*}
where 
$_{i_3,i_2,i_1}\dot J_{k+j_1,k+j_2,k+j_3}=
(\dot J-\{i_1,i_2,i_3\})\cup\{k+j_1,k+j_2,k+j_3\}$. 
By these operators, we obtain 
$
\binom{k}{3}\times \binom{n}{3}
$
functions, which is the number of 
the way to choose $i_1,i_2,i_3,j_1,j_2,j_3$ such that 
$1\le i_1<i_2<i_3\le k,\ 1\le j_1<j_2<j_3\le n$. 

Generally, we have 
\begin{align}
  \label{eq:nabla}
  (\naa_{i_l j_l}\circ \cdots \circ \naa_{i_1j_1})(\f\la \dot J\ra)
  =\frac{\a_{k+j_1}\cdots \a_{k+j_l}}
  {|\tx\la \under{i_l,\dots,i_1}{\dot J}{k+j_1,\dots,k+j_l}\ra|}
  \f\la \under{i_l,\dots,i_1}{\dot J}{k+j_1,\dots,k+j_l}\ra,
\end{align}
where 
$_{i_l,\dots,i_1}\dot J_{k+j_1,\dots,k+j_l}=
(\dot J-\{i_1,\dots,i_l\})\cup\{k+j_1,\dots,k+j_l\}$. 
Note that 
$$0\in \under{i_l,\dots,i_1}{\dot J}{k+j_1,\dots,k+j_l},\quad 
k+n+1\notin \under{i_l,\dots,i_1}{\dot J}{k+j_1,\dots,k+j_l}.$$
In this way, we have 
$$\sum_{i=0}^k 
\binom{k}{i}\times \binom{n}{i}=\binom{k+n}{k}
$$
functions. 
The set of $_{i_l,\dots,i_1}\dot J_{k+j_1,\dots,k+j_l}$'s coincides with the set 
$$\dot \cJ=
\under{k+n+1}{\cJ}{0}=\{J=\{0,j_1,\dots,j_k\}\mid 1\le j_1<\cdots <j_k\le k+n\}.$$
Recall that they are aligned lexicographically
$\dot{\cJ} =\{ J^1=\dot{J} ,J^2,\ldots ,J^r \}$, $r=\binom{k+n}{k}$. 
Recall also that 
$J^p$ is expressed as $\{ 0,j_1 \ldots ,j_k \}$  
with $1\le j_1<\cdots <j_k\le n+k$. 
Note that if $J^p =\under{i_l,\dots,i_1}{\dot J}{k+j_1,\dots,k+j_l}$ as sets, 
we have
$$
\f \la J^p \ra 
= {\rm sgn}\binom{J^p}{\under{i_l,\dots,i_1}{\dot J}{k+j_1,\dots,k+j_l}}
\cdot \f \la \under{i_l,\dots,i_1}{\dot J}{k+j_1,\dots,k+j_l} \ra .
$$
For example, if $k=2$, we have 
$\f \la 023 \ra =-\f \la 032 \ra =-\f \la \under{1}{\dot{J}}{3} \ra $.

We define 
a vector-valued function 
$\bF(\a; x)$ by 
\begin{align*}
  \bF(\a; x)
  = \begin{pmatrix}
     \int_{\square} \prod_{j=1}^{k+n+1} L_j^{\a _j} \f \la \dot{J} \ra \\
     \int_{\square} \prod_{j=1}^{k+n+1} L_j^{\a _j}\f \la J^2 \ra \\
     \vdots \\ 
     \int_{\square} \prod_{j=1}^{k+n+1} L_j^{\a _j}\f \la J^r \ra 
   \end{pmatrix}
= \int_{\square} \prod_{j=1}^{k+n+1} L_j^{\a _j} \cdot \Phi .  
\end{align*}
By (\ref{eq:pd-GMpd}) and (\ref{eq:nabla}), 
it is expressed as 
$$\bFax=G(\a;x)\tilde{\bF}(\a;x),$$ 
where 
\begin{align*}
  \tilde{\bF}(\a;x) &= \tr \Bigl( 
  F(\a;x) , \dfrac{\pa F(\a ;x)}{\pa x_{k1}} ,\ldots ,
  \dfrac{\pa^l F(\a ;x)}{\pa x_{i_1 j_1}\cdots \pa x_{i_l j_l}} ,\ldots 
  \Bigr) ,\\
  G(\a;x)&=\diag\left( 1, 
  \dfrac{x_{k1}}{\a_{k+1}}, 
  \ldots ,
  \pm \dfrac{|\tx \la \under{i_l ,\cdots ,i_1}{\dot{J}}{k+j_1 ,\ldots ,k+j_l} \ra |}
  {\prod_{s=1}^l \a_{k+j_s}} ,\ldots  \right). 
\end{align*}
Note that $G(\a;x)$ belongs to $GL(r;\W^0(X)).$
We call 
the vector-valued function $\bFax$  
(resp. $\tilde{\bF}(\a;x)$) 
the {\it Gauss-Manin vector} of $\Fax$ with respect to 
the frame $\{\f\la J\ra\}_{J\in \dot{\cJ}}$ 
(resp. $G(\a;x)^{-1}\{\f\la J\ra\}_{J\in \dot{\cJ}}$), 
or shortly the {\it G-M vector} of $\Fax$. 
Here, $G(\a;x)^{-1}\{\f\la J\ra\}_{J\in \dot{\cJ}}$ means 
the frame corresponding to the vector $G(\a;x)^{-1}\Phi$.
\begin{theorem}
\label{th:Pfaffian}
The G-M vector $\bFax$ of $\Fax$ satisfies the Pfaffian equation 
$$\dx \bFax=\Psi(\a;x) \bFax$$
with the same connection matrix as 
in (\ref{eq:pfaffian-expression}).
\end{theorem}
\begin{proof}
Since $\bFax$ is defined as $\int_{\square} \prod_{j=1}^{k+n+1} L_j^{\a _j} \cdot \Phi$, 
it satisfies $\dx\bFax=\Psi(\a;x)\bFax$. 
It is clear that this equation is equivalent to a rank $\binom{k+n}{k}$ 
system of differential equations annihilating $\Fax$.
\end{proof}

\begin{cor}
\label{cor:Pfaffian-diffops}
The G-M vector $\tilde{\bF}(\a;x)$ satisfies the Pfaffian equation 
$$\dx \tilde{\bF}(\a;x)=\tilde{\Psi}(\a;x) \tilde{\bF}(\a;x)$$
with the connection matrix 
$$\tilde{\Psi}(\a;x) =G(\a;x)^{-1}\Psi(\a;x)G(\a;x)+\dx G(\a;x)^{-1} G(\a;x).$$
\end{cor}
\begin{proof}
We see the expression of $\tilde{\Psi}(\a;x)$. 
By Theorem \ref{th:Pfaffian}, the G-M vector 
$\tilde{\bF}(\a;x)=G(\a,x)^{-1}\bFax$ satisfies 
\begin{align*}
\dx \tilde{\bF}&=\dx(G^{-1}\bF)
=(\dx G^{-1})\bF+G^{-1}\dx\bF
\\
=&(\dx G^{-1})(GG^{-1})\bF+G^{-1}\Psi\bF
=[\dx G^{-1}G +G^{-1}\Psi G]\tilde{\bF},
\end{align*}
where $\bF=\bFax$, $\Psi=\Psi(\a,x)$ and $G=G(\a,x)$.
\end{proof}

\section{Contiguity relations}
\label{sec:contiguity}
In this section, we give contiguity relations of $\Fax$ 
by using linear maps on twisted cohomology groups and intersection forms. 

For $i=1,\ldots ,k+n+1$, we consider a linear map
$$
\W^l (T_x) \ni \f \mapsto L_i \cdot \f \in \W^l (T_x) .
$$
We put $\a^{(i)}:=(\a_0-1, \ldots ,\a_{i-1},\a_i+1,\a_{i+1} ,\ldots ,\a_{k+n+1})$,  
$\w^{(i)}:=\w +\dt \log L_i$, and 
$\na^{\a^{(i)}} =\dt +\w^{(i)}\we $.
\begin{notation}
  In this section, we write 
  \begin{align*}
    &\cohom =H^k (\Omd (T_x),\na^{\a}) ,& 
    &\cohomi =H^k (\Omd (T_x),\na^{\a^{(i)}}) ,\\
    &\cohom^{\vee}=H^k (\Omd (T_x),\na^{-\a}) ,& 
    &\cohomi^{\vee}=H^k (\Omd (T_x),\na^{-\a^{(i)}}) ,&
  \end{align*}
  for simplicity. 
  For a given $\psi \in \W^k (T_x)$, 
  to clarify which cohomology group $\psi$ belongs to, 
  we denote by $[\psi]$, $[\psi]_i$, $[\psi]^{\vee}$, and $[\psi]_i^{\vee}$
  the element of $\cohom$, $\cohomi$, $\cohom^{\vee}$, and $\cohomi^{\vee}$ 
  represented by $\psi$, respectively.
\end{notation}
\begin{proposition}
  \label{prop:linear-map}
  The map 
  \begin{align*}
    \Conti{i} : \cohomi \ni [\f ]_i \mapsto [L_i \cdot \f ] \in \cohom 
  \end{align*}
  is a well-defined linear map. 
\end{proposition}
\begin{proof}
  Since we have 
  \begin{align*}
    L_i \na^{\a^{(i)}} (\f )
    &=L_i \cdot (\dt \f +\w^{(i)} \we \f)
    =L_i \cdot \left( \dt \f +\w \we \f +\frac{\dt L_i}{L_i}\we \f \right) \\
    &=\dt (L_i \cdot \f )+ \w \we (L_i \cdot \f )
    =\na^{\a}(L_i \f),  
  \end{align*}
  the $L_i$-multiplication descends to a map from $\cohomi$ to $\cohom$. 
\end{proof}
Let $\conti{i} (\a;x)$ be the representation matrix of $\Conti{i}$ 
with respect to the bases 
$\{ [\f \la J \ra ]_i \}_{J\in \dot{\cJ}}$ of $\cohomi$ 
and $\{ [\f \la J \ra ] \}_{ J\in \dot{\cJ}}$ of $\cohom$.
Recall that $\bF (\a ;x)$ is a vector valued function defined as
$$
\tr \Bigl( \int_{\square}\prod_{j=1}^{n+k+1} L_j^{\a_j} \f\la \dot J\ra ,
\int_{\square}\prod_{j=1}^{n+k+1} L_j^{\a_j} \f\la J^2 \ra ,\ldots ,
\int_{\square}\prod_{j=1}^{n+k+1} L_j^{\a_j} \f\la J^r \ra
\Bigr) .
$$
Since 
\begin{align*}
  \int_{\square} \prod_{j=1}^{k+n+1} L_j^{\a _j} 
  \cdot \Conti{i} ([\f \la \dot{J} \ra ]_i)
  =\int_{\square} \prod_{j=1}^{k+n+1} L_j^{\a _j} 
  \cdot L_i \cdot \f \la \dot{J} \ra  
  =F(\a^{(i)};x), 
\end{align*}
we have the contiguity relation
$$
\bF (\a^{(i)};x)=\conti{i} (\a;x) \cdot \bF (\a;x) .
$$
We give an explicit expression of $\conti{i} (\a;x)$. 

\begin{theorem}
  \label{th:contiguity}
  The representation matrix $\conti{i}(\a;x)$ admits the expression  
  \begin{align*}
    \conti{i} (\a;x)=C(\a^{(i)})P_i (\a^{(i)})^{-1} D_i(x) Q_i (\a) C(\a)^{-1}, 
  \end{align*}
  where 
  \begin{align*}
    &D_i(x)=\diag \left( \ldots,\frac{|\tx \la J \ra |}{|\tx \la \under{i}{J}{0} \ra |} 
      ,\ldots \right)_{J\in\under{0}{\cJ}{i}} ,&
    &C(\a)=\Bigl( \cI (\f \la I \ra , \f \la J \ra )\Bigr)_{I,J\in \dot{\cJ}} , \\
    &P_i(\a)=\Bigl( \cI (\f \la I \ra , \f \la J \ra ) 
    \Bigr)_{I\in \under{0}{\cJ}{i},J\in \dot{\cJ}} ,& 
    &Q_i(\a)=\Bigl( \cI (\f \la I \ra , \f \la J \ra ) 
    \Bigr)_{I\in \under{i}{\cJ}{0},J\in \dot{\cJ}} .
  \end{align*}
\end{theorem}

\begin{remark}
  \label{rem:arraying}
  \begin{enumerate}
  \item \label{arraying(1)}
    We explain how to align the elements of 
    $\under{k+n+1}{\cJ}{0}(=\dot{\cJ})$, $\under{0}{\cJ}{i}$, and $\under{i}{\cJ}{0}$. 
    First, recall that we align the elements of $\under{k+n+1}{\cJ}{0}$ lexicographically, 
    and denote them $\under{k+n+1}{\cJ}{0} =\{ J^1 ,\ldots ,J^r \}$, 
    where $r=\binom{k+n}{k}$. 
    Next, we align the elements of $\under{i}{\cJ}{0}$ as 
    \begin{align*}
      \under{i}{\cJ}{0} =\{ {J'}^1 ,\ldots , {J'}^r \} ,\quad
      {J'}^l =\left\{
        \begin{array}{ll}
          J^l & (i \not\in J^l), \\
          \under{i}{J^l}{k+n+1} & (i \in J^l). 
        \end{array}
      \right.
    \end{align*}
    Finally, we align the elements of $\under{0}{\cJ}{i}$ as 
    $$
    \under{0}{\cJ}{i} =\{ \under{0}{{J'}^1}{i} ,\ldots , \under{0}{{J'}^r}{i} \} .
    $$
    For example, if $k=n=2$ and $i=3$, then 
    \begin{align*}
      \under{5}{\cJ}{0}&=\{ \{012\},\{013\},\{014\},\{023\},\{024\},\{034\} \}, \\
      \under{3}{\cJ}{0}&=\{ \{012\},\{015\},\{014\},\{025\},\{025\},\{054\} \}, \\
      \under{0}{\cJ}{3}&=\{ \{312\},\{315\},\{314\},\{325\},\{325\},\{354\} \}. 
    \end{align*}
  \item 
    The intersection numbers $\cI (\f \la I \ra , \f \la J \ra )$ in the theorem
    can be computed by Fact \ref{fact:intersection}. 
    We denote the intersection pairing between 
    $\cohomi$ and $\cohomi^{\vee}$
    by $\cI^{(i)}$. Then, it is easy to see that 
    \begin{align*}
      \Bigl( \cI^{(i)} ([\f \la I \ra ]_i, [\f \la J \ra ]_i^{\vee} )
      \Bigr)_{I,J\in \dot{\cJ}}
      &=C(\a^{(i)}),\\ 
      \Bigl( \cI^{(i)} ([\f \la I \ra ]_i, [\f \la J \ra ]_i^{\vee} )
      \Bigr)_{I\in \under{0}{\cJ}{i},J\in \dot{\cJ}}
      &=P_i(\a^{(i)}). 
    \end{align*}
  \end{enumerate}
\end{remark}

\begin{proof}[Proof of Theorem \ref{th:contiguity}]
  Recall that $L_0=1$. 
  For $J \in \under{0}{\cJ}{i}$, because of  
  \begin{align*}
    \f \la J \ra = |\tx \la J \ra |\cdot \frac{\dd t}{\prod_{j\in J} L_j}
    = |\tx \la J \ra |\cdot \frac{\dd t}{L_i \cdot \prod_{j\in \under{i}{J}{0}} L_j} ,
  \end{align*}
  we have
  \begin{align*}
    L_i \cdot \f \la J \ra 
    = |\tx \la J \ra |\cdot \frac{\dd t}{\prod_{j\in \under{i}{J}{0}} L_j}
    =\frac{|\tx \la J \ra |}{|\tx \la \under{i}{J}{0} \ra |} \cdot \f \la \under{i}{J}{0} \ra .
  \end{align*}
  Hence, the alignment mentioned in Remark \ref{rem:arraying} (\ref{arraying(1)}) means that 
  the representation matrix of $\Conti{i}$ with respect to the bases
  \begin{align*}
    \{ [\f \la J \ra ]_i \}_{ J\in \under{0}{\cJ}{i}} \subset \cohomi ,\quad 
    \{ [\f \la J \ra ]  \}_{J\in \under{i}{\cJ}{0}} \subset \cohom 
  \end{align*}
  coincides with $D_i (x)$. 
  By linearity of the intersection forms $\cI$ and $\cI^{(i)}$, 
  we can show that 
  \begin{align*}
    \begin{pmatrix}
      [\f \la {J'}^1 \ra ] \\ \vdots \\ [\f \la {J'}^r \ra ] 
    \end{pmatrix}
    &=Q_i(\a)C(\a)^{-1}
    \begin{pmatrix}
      [\f \la J^1 \ra ] \\ \vdots \\ [\f \la J^r \ra ] 
    \end{pmatrix} ,\\ 
    \begin{pmatrix}
      [\f \la \under{0}{{J'}^1}{i} \ra ]_i \\ \vdots \\ [\f \la \under{0}{{J'}^r}{i} \ra ]_i 
    \end{pmatrix}
    &=P_i(\a^{(i)})C(\a^{(i)})^{-1}
    \begin{pmatrix}
      [\f \la J^1 \ra ]_i \\ \vdots \\ [\f \la J^r \ra ]_i 
    \end{pmatrix} .
  \end{align*}
  These imply that the representation matrix $\conti{i}(\a;x)$ 
  coincides with $C(\a^{(i)})P_i (\a^{(i)})^{-1} D_i(x) Q_i (\a) C(\a)^{-1}$. 
\end{proof}

In the remainder of this section, we consider relations between 
the linear map $\Conti{i}$ and the intersection form $\cI$. 
\begin{theorem}
  \label{th:contiguity-intersection}
  The linear map 
  $\Conti{i} : \cohomi \to \cohom$ is 
  expressed as 
  \begin{align}
    \label{eq:expression}
    \Conti{i} ([\f ]_i) &=\sum_{J\in \under{0}{\cJ}{i}}
    \frac{\cI^{(i)} ([\f ]_i, [\f \la \under{i}{J}{0} \ra ]_i^{\vee})}
    {\cI^{(i)} ([\f \la J \ra ]_i , [\f \la \under{i}{J}{0} \ra ]_i^{\vee} )}
    \cdot \frac{|\tx \la J \ra |}{|\tx \la \under{i}{J}{0} \ra |} 
    \cdot [\f \la \under{i}{J}{0} \ra ] \\
    \nonumber
    &=\sum_{J\in \under{i}{\cJ}{0}}
    \frac{\cI^{(i)} ([\f ]_i, [\f \la J \ra ]_i^{\vee})}
    {\cI^{(i)} ([\f \la \under{0}{J}{i} \ra ]_i , [\f \la J \ra ]_i^{\vee} )}
    \cdot \frac{|\tx \la \under{0}{J}{i} \ra |}{|\tx \la J \ra |} 
    \cdot [\f \la J \ra ] .
  \end{align}
\end{theorem}
\begin{proof}
  Since a correspondence 
  $\under{0}{\cJ}{i} \ni J \mapsto \under{i}{J}{0} \in \under{i}{\cJ}{0}$ is 
  one-to-one, the second equality is clear. 
  We show the first one. 
  By Fact \ref{fact:intersection}, we have 
  \begin{align*}
    \cI^{(i)} ([\f \la J \ra ]_i, [\f \la \under{i}{J'}{0} \ra ]_i^{\vee}) 
    =\left\{ 
      \begin{array}{cl}
        \DS \frac{(\tpi)^k}{\prod_{j\in J-\{ i\}} \a_j} & 
        {\rm if}\  J=J' ,\\
        0 & {\rm otherwise}, 
      \end{array}
    \right.
  \end{align*}
  for $J,J'\in \under{0}{\cJ}{i}$. 
  If $\f =\f \la J \ra $ with $J\in \under{0}{\cJ}{i}$, 
  then the right-hand side of (\ref{eq:expression}) is 
  $$
  \frac{|\tx \la J \ra |}{|\tx \la \under{i}{J}{0} \ra |} 
  \cdot [\f \la \under{i}{J}{0} \ra ] 
  $$
  which is nothing but $\Conti{i} ([\f \la J \ra ]_i)$, 
  by the proof of Theorem \ref{th:contiguity}. 
  Since $\{ [\f \la J \ra ]_i \}_{J\in \dot{\cJ} }$  
  form a basis of $\cohomi$, 
  the first equality holds. 
\end{proof}
In a way similar to that used in Proposition \ref{prop:linear-map}
and Theorem \ref{th:contiguity-intersection}, 
we can show the following. 
\begin{cor}
  The inverse map of $\Conti{i}$ is given by a well-defined map 
  $$
  \Conti{i}^{-1} : \cohom \ni [\f] \mapsto 
  \Big[ \frac{1}{L_i}\cdot \f \Big]_i \in \cohomi .
  $$
  It also admits the expression 
  \begin{align}
    \label{eq:expression2}
    \Conti{i}^{-1} ([\f ])
    =\sum_{J\in \under{0}{\cJ}{i}}
    \frac{\cI ([\f ], [\f \la J \ra ]^{\vee})}
    {\cI ([\f \la \under{i}{J}{0} \ra ] , [\f \la J \ra ]^{\vee} )}
    \cdot \frac{|\tx \la \under{i}{J}{0} \ra |}{|\tx \la J \ra |} 
    \cdot [\f \la J \ra ]_i .
  \end{align}
\end{cor}
\begin{remark}
  Expressions similar to (\ref{eq:expression}) and (\ref{eq:expression2}) 
  are given in \cite[\S 4.4.2]{AK} without the intersection forms. 
  Their calculations are done in only the target space $\cohom$ (resp. $\cohomi$) 
  of $\Conti{i}$ (resp. $\Conti{i}^{-1}$), and are complicated. 
  On the other hand, by considering not only the target spaces 
  but also the domains of $\Conti{i}$ and $\Conti{i}^{-1}$,  
  we can obtain a simple structure of 
  contiguity relations. 
\end{remark}

Replacing $\a$ by $-\a^{(i)}$ in Proposition \ref{prop:linear-map}, 
we obtain the linear map 
\begin{align*}
  \Conti{i}^{\vee} : 
  \cohom^{\vee} \ni [\f ]^{\vee} \mapsto [L_i \cdot \f ]_i^{\vee} \in \cohomi^{\vee}.
\end{align*}
\begin{proposition}
  For any $[\f ]_i \in \cohomi$ and $[\psi]^{\vee} \in \cohom^{\vee}$, 
  we have
  \begin{align*}
    \cI \bigl( \Conti{i}([\f ]_i), [\psi]^{\vee} \bigr)
    = \cI^{(i)} \bigl( [\f ]_i , \Conti{i}^{\vee} ([\psi ]^{\vee}) \bigr) . 
  \end{align*}
\end{proposition}
\begin{proof}
  By \cite{M1}, there exist $C^{\infty}$ $k$-forms $\f'$ and $\eta$ on $T_x$
  such that the support of $\f'$ is compact and
  $$
  \f =\f' +\na^{\a^{(i)}} \eta . 
  $$
  By the proof of Proposition \ref{prop:linear-map}, we have 
  $$
  L_i \cdot \f =L_i \cdot \f' +\na^{\a} (L_i \cdot \eta ) .
  $$
  Since the support of $L_i \cdot \f'$ is also compact, 
  the intersection numbers are expressed as 
  \begin{align*}
    \cI \bigl( \Conti{i}([\f ]_i), [\psi]^{\vee} \bigr)
    &=\int_{T_x} (L_i \cdot \f') \wedge \psi ,\\
    \cI^{(i)} \bigl( [\f]_i , \Conti{i}^{\vee} ([\psi]^{\vee}) \bigr)
    &=\int_{T_x} \f' \wedge (L_i \cdot \psi ).
  \end{align*}
  Obviously, these two integrations coincide. 
\end{proof}
By considering the bases of 
$\cohom$, $\cohom^{\vee}$, $\cohomi$, and $\cohomi^{\vee}$
represented by $\{ \f \la J \ra \}_{J\in \dot{\cJ}}$, 
we obtain the following identity. 
\begin{cor}
  $\conti{i} (\a ;x) C(\a) =C(\a^{(i)}) \tr \conti{i} (-\a^{(i)} ;x)$. 
\end{cor}

\section{Relations for hypergeometric series}
\label{sec:series}
For applications to algebraic statistics, we need to reduce our formulas 
for hypergeometric integrals to those 
for hypergeometric series. 
In this section, we specialize a twisted cycle $\square$ to 
$\De\la 1,\dots,k,k+n+1\ra$, which is the regularization of 
the standard simplex (\ref{eq:s-simplex}) with respect to 
$\prod_{j=1}^{k+n+1} L_j^{\a_j}$.
Then the integral $\Fax$ admits a power series expansion for $x$ 
sufficiently close to the zero matrix $O$.
We give relations between this expansion and the hypergeometric series 
defined in \cite{AK}.

We put 
\begin{align*}
  S(\a;x)=
  \sum_{m=(m_{ij})\in M(k,n;\Z_{\geq 0})}
  \frac{1}{\G_m(\a)} \cdot \prod_{i,j}x_{ij}^{m_{ij}} ,
\end{align*}
where 
\begin{align*}
  \G_m(\a)=&
  \prod_{i=1}^k \G(-\a_i -\sum_{j=1}^n m_{ij} +1)
  \cdot \prod_{j=1}^n \G(\a_{k+j}-\sum_{i=1}^k m_{ij}+1) \\
  &\cdot \G(\sum_{i=1}^{k}\a_i+\a_{k+n+1}+\sum_{i=1}^k \sum_{j=1}^n m_{ij}+1)
  \cdot \prod_{i=1}^k \prod_{j=1}^n \G(m_{ij}+1).
\end{align*}
Note that $S(\a;x)$ coincides with the hypergeometric series 
of type $(k+1,k+n+2)$ defined in \cite[\S 3.1.3]{AK}, modulo gamma factors. 
\begin{proposition}
  We specialize a twisted cycle $\square$ to 
  $\De\la 1,\dots,k,k+n+1\ra$. 
  If each $x_{ij}$ is sufficiently close to $0$, 
  then the integral $F(\a ;x)$ admits the power series expansion
  \begin{align*}
    e^{-\pi \sqrt{-1}(\a_1 +\cdots +\a_k)} \cdot 
    \prod_{i=1}^k \G(\a_i)\G(-\a_i+1) \cdot 
    \prod_{j=1}^{n+1}\G(\a_{k+j} +1) \cdot S(\a ;x) .
  \end{align*}
\end{proposition}
\begin{proof} 
  We give the arguments of $L_i$ on the standard simplex $\De$ in 
(\ref{eq:s-simplex}) as follows.
  \begin{align*}
    \begin{array}{|c|c|c|c|}
      \hline
      &i=1,\ldots ,k & i=k+1,\ldots ,k+n & i=k+n+1 \\ \hline
      \arg L_i &-\pi & 0 & 0 \\ \hline
    \end{array}
  \end{align*}
  By putting $t_i=e^{-\pi \sqrt{-1}}s_i (=-s_i)$ in 
  the integration
  \begin{align*}
    F(\a ;x)=\int_{\De\la 1,\dots,k,k+n+1\ra} \prod_{j=1}^{k+n+1} L_j^{\a_j} \f \la \dot{J} \ra ,
  \end{align*}
  we can show the proposition in an analogous way used in \cite[\S 3.3]{AK}.
\end{proof}

We put 
\begin{align*}
  \bS(\a;x)&=
  e^{\pi \sqrt{-1}(\a_1 +\cdots +\a_k)} \cdot 
  \prod_{i=1}^k \frac{1}{\G(\a_i)\G(-\a_i+1)} \cdot 
  \prod_{j=1}^{n+1}\frac{1}{\G(\a_{k+j} +1)} \cdot \bF(\a, x) \\
  &=
  \begin{pmatrix}
    S(\a;x) \\
    \vdots \\ 
    \pm \dfrac{|\tx \la \under{i_l ,\cdots ,i_1}{\dot{J}}{k+j_1 ,\ldots ,k+j_l} \ra |}
    {\prod_{s=1}^l \a_{k+j_s}}
    \cdot \dfrac{\pa^l S(\a ;x)}{\pa x_{i_1 j_1}\cdots \pa x_{i_l j_l}}  \\
    \vdots
  \end{pmatrix}, 
\end{align*}
which is the G-M vector of $S(\a;x)$. 
We consider the Pfaffian equation and 
the contiguity relations with respect to $\bS (\a;x)$. 

\begin{cor}
  \label{cor:pfaffian}
  $\dx \bS(\a;x)=\Psiax \bS(\a;x)$.
\end{cor}
\begin{proof}
  $\bS (\a;x)$ is defined as a scalar multiple of $\bF (\a;x)$ and 
  this scalar is independent of $x$. 
\end{proof}

\begin{cor}
  \label{cor:contiguity}
  For $1\leq i \leq k$ and $1\leq j \leq n+1$, we have
  \begin{align*}
    \bS(\a^{(i)};x)&=\conti{i}(\a;x)\bS(\a;x), \\
    \bS(\a^{(k+j)};x)&=\frac{1}{\a_{k+j}+1}\conti{k+j}(\a;x)\bS(\a;x) .
  \end{align*}
\end{cor}
\begin{proof}
  By
  \begin{align*}
    (-1)\cdot \frac{\G(\a_i)\G(-\a_i+1)}{\G(\a_i+1)\G(-\a_i)} =1, \quad 
    \frac{\G(\a_{k+j} +1)}{\G(\a_{k+j} +2)}=\frac{1}{\a_{k+j} +1} ,  
  \end{align*}
  and Theorem \ref{th:contiguity}, we have the identities. 
\end{proof}

\section{Normalizing constants of two-way contingency tables}
\label{sec:NC}
We apply contiguity relations and the Pfaffian equation 
to the numerical evaluation of 
the normalizing constants of 
the hypergeometric distribution of 
the $r_1 \times r_2$ contingency tables 
with fixed marginal sums. 
In this section, we explain how our results are applied, 
and give an algorithm that evaluates the normalizing constants. 

We consider an $r_1 \times r_2$ contingency table 
\begin{align*}
  u=
  \begin{array}{|c|c|c|c|c}
    \cline{1-4} 
    u_{11} & u_{12} & \cdots & u_{1r_2} &\b_{1}^{(1)} \\ \cline{1-4} 
    u_{21} & u_{22} & \cdots & u_{2r_2} &\b_{2}^{(1)} \\ \cline{1-4} 
    \vdots & \vdots &        & \vdots &\vdots \\ \cline{1-4} 
    u_{r_1 1} & u_{r_1 2} & \cdots & u_{r_1 r_2} &\b_{r_1}^{(1)} \\  \cline{1-4} 
    \multicolumn{1}{c}{\b_{1}^{(2)}}&\multicolumn{1}{c}{\b_{2}^{(2)}} 
    &\multicolumn{1}{c}{\cdots} & \multicolumn{1}{c}{\b_{r_2}^{(2)}} & 
  \end{array} 
  \qquad u_{ij} \in \Z_{\geq 0} .
\end{align*}
Here, $\b^{(1)}_i :=\sum_{j=1}^{r_2}u_{ij}$ is the row sum, and  
$\b^{(2)}_j :=\sum_{i=1}^{r_1}u_{ij}$ is the column sum. 
For fixed marginal sums 
$\b=(\b^{(1)};\b^{(2)})=(\b_1^{(1)},\ldots ,\b_{r_1}^{(1)}; \b_1^{(2)},\ldots ,\b_{r_2}^{(2)})$ 
and a variable matrix $p=(p_{ij})_{1\leq i \leq r_1,1\leq j \leq r_2}$,  
the polynomial 
$$
Z(\b ;p)= \sum_u \frac{p^u}{u!} 
=\sum_u \frac{\prod_{i,j} p_{ij}^{u_{ij}}}{\prod_{i,j} u_{ij} !}
=\sum_u \frac{\prod_{i,j} p_{ij}^{u_{ij}}}{\prod_{i,j} \G (u_{ij}+1)}
$$
in $p_{ij}$ is called the normalizing constant, 
where the sum is taken over all contingency tables $u$ with 
marginal sums $\b$. 

\begin{proposition}
  \label{prop:NCtoS}
  We put the parameters and variables in $S(\a;x)$ as follows: 
  \begin{align*}
    (k,n)&:=(r_1-1,r_2-1), \\
    \a &=(\a_0,\ldots ,\a_{k+n+1})
    =(\a_0,\a_1 ,\ldots ,\a_{r_1-1},\a_{r_1}, \ldots ,\a_{r_1+r_2-2} ,\a_{r_1+r_2-1})\\
    &:=(-\b_{r_1}^{(1)}, -\b_1^{(1)},\ldots ,-\b_{r_1-1}^{(1)}, 
    \b_2^{(2)},\ldots ,\b_{r_2}^{(2)},\b_1^{(2)}), \\
    x&:=(x_{ij})_{1\leq i \leq k,1\leq j \leq n}, \quad x_{ij}=\frac{p_{i,j+1}p_{r_1 1}}{p_{i1}p_{r_1,j+1}}. 
  \end{align*}
  Then the normalizing constant is expressed as 
  \begin{align*}
    Z(\b ;p)=\prod_{i=1}^{k}p_{i1}^{-\a_i} \cdot 
    \prod_{j=1}^{n}p_{r_1,j+1}^{\a_{k+j}} \cdot 
    p_{r_1 1}^{\sum_{i=1}^{k}\a_i +\a_{k+n+1}} 
    \cdot S(\a ;x). 
  \end{align*}
\end{proposition}
Note that $S(\a ;x)$ in the right-hand side is a polynomial in $x$. 
\begin{proof}
  We put 
  \begin{align*}
    u_0:=   
    \begin{array}{|c|c|c|c|}
      \cline{1-4} 
      \b_{1}^{(1)} & 0 & \cdots & 0  \\ \hline 
      \vdots & \vdots &        & \vdots  \\ \hline 
      \b_{r_1-1}^{(1)} & 0 & \cdots & 0  \\ \hline 
      \b_{1}^{(2)}-\sum_{i=1}^{r_1-1}\b_{i}^{(1)} & \b_{2}^{(2)} & \cdots & \b_{r_2}^{(2)}  \\  \hline 
    \end{array} 
    =   
    \begin{array}{|c|c|c|c|}
      \cline{1-4} 
      -\a_{1} & 0 & \cdots & 0  \\ \hline 
      \vdots & \vdots &        & \vdots  \\ \hline 
      -\a_k & 0 & \cdots & 0  \\ \hline 
      \a_{k+n+1}+\sum_{i=1}^{k}\a_{i} & \a_{k+1} & \cdots & \a_{k+n}  \\  \hline 
    \end{array}
    \ ,
  \end{align*}
  \begin{align*}    
    \ell_{ij}:=  
    \begin{array}{|c|c|c|c|c}
      \cline{1-4} 
      \vdots  &        & \vdots & \\ \cline{1-4} 
      -1  & \cdots & 1 &\cdots &i \\ \cline{1-4} 
      \vdots & & \vdots         &   \\ \cline{1-4} 
      1  & \cdots & -1 &\cdots & \\  \cline{1-4} 
      \multicolumn{1}{c}{} 
      &\multicolumn{1}{c}{} & \multicolumn{1}{c}{j+1}& 
      \multicolumn{1}{c}{} & 
    \end{array} \ (\textrm{the other entries are 0}).
  \end{align*}
  The contingency table $u$ with the marginal sums $\b$ is
  expressed as 
  $$
  u=u_0 +\sum_{i=1}^{r_1-1}\sum_{j=1}^{r_2-1} m_{ij} \ell_{ij} ,
  $$
  for some $m=(m_{ij}) \in M(r_1-1,r_2-1;\Z_{\geq 0})=M(k,n;\Z_{\geq 0})$. 
  This $m$ is uniquely determined. 
  If some entries of $u$ are negative integers, 
  then $\prod p_{ij}^{u_{ij}} / \prod \G (u_{ij}+1) =0$, because of 
  $1/ \G(N)=0$ for $N\in \Z_{\leq 0}$.
  We thus have
  \begin{align*}
    Z(\b ;p)
    =\sum_u \frac{p^u}{\prod_{a,b} \G (u_{ab}+1) }
    =\sum_{\substack{u=u_0+\sum_i \sum_j m_{ij}\ell_{ij} \\ m \in M(k,n; \Z )}} 
    \frac{p^u}{\prod_{a,b} \G (u_{ab}+1) }. 
  \end{align*}
  If $u=(u_{ab})=u_0+\sum_i \sum_j m_{ij}\ell_{ij}$, then
  \begin{align*}
    &p^u =p^{u_0} \cdot \prod_{i=1}^k \prod_{j=1}^n x_{ij}^{m_{ij}}
    =\prod_{i=1}^{k}p_{i1}^{-\a_i} \cdot 
    \prod_{j=1}^{n}p_{r_1,j+1}^{\a_{k+j}} \cdot 
    p_{r_1 1}^{\sum_{i=1}^{k}\a_i +\a_{k+n+1}} \cdot 
    \prod_{i=1}^k \prod_{j=1}^n x_{ij}^{m_{ij}}, \\
    &\frac{1}{\prod_{a=1}^{r_1}\prod_{b=1}^{r_2} \G (u_{ab}+1)}
    =\frac{1}{\G_m (\a)} .
  \end{align*}
  Hence, the proposition is proved. 
\end{proof}

Hereafter, we put 
$\a$, $x$, and $u_0$ as in Proposition \ref{prop:NCtoS} and its proof. 
Then the normalizing constant is expressed as 
$Z(\b ;p)=p^{u_0}\cdot S(\a ;x)$. 

According to \cite{TKT}, the expectation $E[U_{ij}]$ of the $(i,j)$-cell is given as
$$
p_{ij} \frac{\pa}{\pa p_{ij}}\log Z
=\frac{1}{Z} \cdot p_{ij} \frac{\pa Z}{\pa p_{ij}} .
$$
It is known that the expectations are functions in $x$ 
(see also Corollary \ref{cor:euler}). 
Further, the values 
$\pa E[U_{ij}] / \pa x_{i'j'}$
are important to solve the conditional maximal likelihood estimate problem. 
We express the expectations and their derivatives by 
entries of $\bS (\a;x)$ and $\dx \bS (\a ;x)$. 
Recall that $\bS(\a;x)$ is aligned by the elements of $\dot{\cJ}=\under{k+n+1}{\cJ}{0}$. 
For $J\in \dot{\cJ}$, 
we call the entry of $\bS(\a;x)$ corresponding to $J$ 
by the {\it $J$-entry}. 
\begin{cor}
  \label{cor:euler}
  For $1\leq i, i' \leq k$ and $1 \leq j ,j' \leq n$, 
  let $S_{ij}$ be 
  the $\under{i}{\dot{J}}{k+j}$-entry of $\bS(\a;x)$, 
  and let $S_{(ij)(i'j')}$ be the $\under{i}{\dot{J}}{k+j}$-entry of
  the coefficient of $\dd x_{i'j'}$ in $\dx \bS (\a;x)=\Psi (\a; x)\bS(\a;x)$. 
  We denote the $(i,j)$-entry of $u_0$ by $(u_0)_{ij}$. 
  Then we have 
  \begin{align*}
    E[U_{ij}] &=(u_0)_{ij} 
    +\frac{(-1)^k}{S}\sum_{a=1}^{k} \sum_{b=1}^{n} (-1)^a (\ell_{ab})_{ij} \a_{k+b} S_{ab},  \\
    \frac{\pa E[U_{ij}]}{\pa x_{i'j'}}
    &=\frac{(-1)^k}{S^2}
    \sum_{a=1}^{k} \sum_{b=1}^{n}(-1)^a (\ell_{ab})_{ij} \a_{k+b}
    \left( S\cdot S_{(ab)(i'j')} -(-1)^{k-i'}\frac{\a_{k+j'}}{x_{i'j'}}\cdot S_{ab}\cdot S_{i'j'}  \right), 
  \end{align*}
  where 
  \begin{align*}
    (\ell_{ab})_{ij} =\left\{
      \begin{array}{cl}
        1 & \textrm{if} \ (i,j)=(a,b+1) \ \textrm{or} \ (r_1 ,1), \\
        -1& \textrm{if} \ (i,j)=(a,1) \ \textrm{or} \ (r_1 ,b+1), \\
        0 & \textrm{otherwise}.
      \end{array}
    \right.
  \end{align*}
\end{cor}
\begin{proof}
  As mentioned in Section \ref{sec:pfaffian}, 
  the $\under{i}{\dot{J}}{k+j}$-entry of $\bS(\a;x)$ is
  \begin{align*}
    S_{ij} &={\rm sgn}
    \begin{pmatrix}
      0 & \cdots & i-1 & i+1 & i+2 &\cdots &k &k+j \\
      0 & \cdots & i-1 & k+j & i+1 &\cdots &k-1&k  
    \end{pmatrix}
    \cdot \frac{|\tx \la \under{i}{\dot{J}}{k+j} \ra|}{\a_{k+j}}\frac{\pa S(\a;x)}{\pa x_{ij}} \\
    &=(-1)^{k-i} \cdot \frac{x_{ij}}{\a_{k+j}}\frac{\pa S(\a;x)}{\pa x_{ij}} .
  \end{align*}
  Since $x_{ab}=p^{\ell_{ab}}=p_{a,b+1}p_{r_1 1}p_{a1}^{-1}p_{r_1,b+1}^{-1}$, we obtain 
  \begin{align*}
    E[U_{ij}]
    &=\frac{1}{Z} \cdot p_{ij}\frac{\pa Z}{\pa p_{ij}} 
    =\frac{1}{p^{u_0} S(\a ;x)} \cdot p_{ij}\frac{\pa}{\pa p_{ij}} (p^{u_0} S(\a ;x)) \\
    &=(u_0)_{ij} 
    +\frac{1}{S}\sum_{a=1}^{k} \sum_{b=1}^{n} p_{ij}\frac{\pa x_{ab}}{\pa p_{ij}}\frac{\pa S}{\pa x_{ab}} \\
    &=(u_0)_{ij} 
    +\frac{1}{S}\sum_{a=1}^{k} \sum_{b=1}^{n} (\ell_{ab})_{ij}\cdot x_{ab} \frac{\pa S}{\pa x_{ab}} \\
    &=(u_0)_{ij} 
    +\frac{1}{S}\sum_{a=1}^{k} \sum_{b=1}^{n} (\ell_{ab})_{ij} \cdot (-1)^{k-a}\a_{k+b} S_{ab} .  
  \end{align*}
  We can easily obtain the second equality from the first one
  by using $S_{(ij)(i'j')}=\pa S_{ij} / \pa x_{i'j'}$. 
\end{proof}
\begin{remark}
  By using this corollary and the chain rule, 
  we can also obtain 
  the gradient $(\pa Z/\pa p_{ij})_{i,j}$ and 
  the Hessian $(\pa^2 Z/\pa p_{ij}\pa p_{i'j'})_{(i,j),(i',j')}$ 
  from $\bS (\a;x)$ and $\dx \bS (\a;x)$. 
\end{remark}
By Proposition \ref{prop:NCtoS} and Corollary \ref{cor:euler}, 
we can reduce the numerical evaluation of the normalizing constant, 
the expectations and their derivatives 
to that of $\bS (\a ;x)$ and $\dx \bS (\a; x)$. 
To apply our results, 
we use the following lemma. 

\begin{lemma}
  For any $\check{\a}^0 =(\a_1^0 ,\ldots ,\a_{k+n+1}^0) \in \C^{k+n+1}$,  
  the series $S(\a ;x) \ (\a_0 =-\sum_{j=1}^{k+n+1} \a_j)$ 
  as a function in 
  $(\a_1 ,\ldots ,\a_{k+n+1} ;x) \in \C^{k+n+1} \times \C^{kn}$ 
  converges uniformly on a small neighborhood of $(\check{\a}^0 ;O)$. 
\end{lemma}
\begin{proof}
  This lemma can be shown in a similar way 
  to \cite[Lemma 7.3]{G-FD}. 
\end{proof}

We put $\a^0:=(-\sum_{j=1}^{k+n+1} \a_j^0, \a_1^0 ,\ldots ,\a_{k+n+1}^0)$. 
If $\a_i^0 \in \Z_{<0}$ ($1\leq i \leq k$) and $\a_{k+j}^0 \in \Z_{>0}$ ($1\leq j \leq n$),
then the series $S(\a^0 ;x)$
becomes a polynomial in $x_{ij}$ ($1\leq i \leq k ,\ 1\leq j \leq n$). 
This lemma implies that 
$\lim_{\a \to \a^0} S(\a ;x)$ coincides with the polynomial $S(\a^0 ;x)$, 
for $x$ in a small neighborhood of the zero matrix $O$. 
In a similar way, we can show that 
the partial derivatives of $S(\a ;x^0)$ 
converge to those of the the polynomial $S(\a^0 ;x)$, as $\a \to \a^0$. 
By the identity theorem for holomorphic functions, 
we obtain the following corollary. 

\begin{cor}
  Let $\a$ be the integer vector defined in Proposition \ref{prop:NCtoS}.  
  Then the relations in 
  Corollaries \ref{cor:pfaffian} and \ref{cor:contiguity} hold, 
  as those between the vectors consisting of polynomials in $x$. 
\end{cor}

By Proposition \ref{prop:inverse} in Appendix \ref{appendix}, 
we obtain the following lemma. 
\begin{lemma}
  \label{lem:invertible}
  If none of $\a_i \ (0\leq i \leq k+n+1)$ and 
  $|\tx \la J \ra| \ (J \in \cJ)$ is zero, 
  then the matrices $\Psiax$ and $\conti{i} (\a ;x)$ are well-defined. 
  Further, $\conti{i} (\a ;x)$ is invertible. 
\end{lemma}
We explain an algorithm to evaluate the normalizing constant 
by using the contiguity relations. 
We put 
\begin{align*}
  &\a^0:=(1-r_2 ,\underbrace{-1,\ldots ,-1}_{r_1-1} ,\underbrace{1,\ldots ,1}_{r_2-1},r_1-1), \\
  &\de_i :=(\underset{0}{-1} ,\underset{1}{0}, \ldots ,
  \underset{i-1}{0},\underset{i}{1},\underset{i+1}{0},\ldots ,\underset{r_1+r_2-1}{0}), \quad 
  i=1,\ldots ,r_1+r_2-1.
\end{align*}
\begin{algorithm}\label{algo:path}~
  \begin{itemize}
  \item[Input:] a parameter vector $\a =(\a_0 ,\ldots ,\a_{r_1+r_2-1})$ with 
    $$
    \a_{i} \in \Z_{<0} \ (0\leq i \leq r_1-1), \quad 
    \a_{r_1+j} \in \Z_{>0} \ (0\leq j \leq r_2-1) .
    $$ 
  \item[Output:] 
    a sequence $\{ \a^l \}_{l=1}^{e}$ satisfying 
    \begin{enumerate}[{\rm (i)}]
    \item $\a^e=\a$, 
    \item each entry of $\a^l =(\a^l_0 ,\ldots ,\a^l_{r_1+r_2-1})$ is nonzero, 
    \item for $1\leq l \leq e$, the difference $\a^l -\a^{l-1}$ is one of the following: 
      \begin{align*}
        -\de_1 ,\ldots ,-\de_{r_1-1} , 
        \de_{r_1} ,\ldots ,\de_{r_1+r_2-2} ,
        \pm \de_{r_1+r_2-1}. 
      \end{align*}
    \end{enumerate}
  \end{itemize}
  \begin{enumerate}[1.]
  \item Let $\si =(\si_0,\ldots ,\si_{r_1+r_2-1}):=\a -\a^0$. 
    Then $\si_1 ,\ldots, \si_{r_1-1} \in \Z_{\leq 0}$ and 
    $\si_{r_1} ,\ldots, \si_{r_1+r_2-2} \in \Z_{\geq 0}$. 
  \item Let $l:=0$. For $j$ from $0$ to $r_2-2$, 
    while $\a^l_{r_1+j}<\a_{r_1+j}$, do
    \begin{align*}
      \a^{l+1}&\leftarrow \a^l +\de_{r_1+j}, \\
      l& \leftarrow l+1 .
    \end{align*}
  \item If $r_1-1<\a_{r_1+r_2-1}$, 
    while $\a^l_{r_1+r_2-1}<\a_{r_1+r_2-1}$, do
    \begin{align*}
      \a^{l+1}&\leftarrow \a^l +\de_{r_1+r_2-1}, \\
      l& \leftarrow l+1 . 
    \end{align*}
    Else if $r_1-1>\a_{r_1+r_2-1}$, while $\a^l_{r_1+r_2-1}>\a_{r_1+r_2-1}$, do
    \begin{align*}
      \a^{l+1}&\leftarrow \a^l -\de_{r_1+r_2-1}, \\
      l& \leftarrow l+1. 
    \end{align*}
  \item For $i$ from $1$ to $r_1-1$, 
    while $\a^l_{i}>\a_{i}$, do
    \begin{align*}
      \a^{l+1}&\leftarrow \a^l -\de_{i}, \\
      l& \leftarrow l+1 .
    \end{align*}
  \item Return $\a^1 ,\ldots ,\a^{l-1}(=:\a^e )$. 
  \end{enumerate}
\end{algorithm}
By (ii) and Lemma \ref{lem:invertible}, 
the matrix $\conti{i}(\a^l ;x)$ is well-defined and invertible  
for $1\leq i \leq r_1+r_2-1$ and $0\leq l \leq e$. 

\begin{algorithm}\label{algo:NC}~
  \begin{itemize}
  \item[Input:] marginal sums $\b$ and probabilities $p$. 
  \item[Output:] the normalizing constant $Z(\b;p)$, 
    the expectations $E[U_{ij}]$, and 
    their derivatives $\pa E[U_{ij}]/ \pa x_{i'j'}$. 
  \end{itemize}
  \begin{enumerate}[1.]
  \item Let $\a$ and $x$ be as in Proposition \ref{prop:NCtoS}. 
  \item By using Algorithm \ref{algo:path}, 
    find a sequence $\{ \a^l \}_{i=1}^e$ satisfying (i), (ii), and (iii). 
  \item Compute $\bS (\a^0 ;x)$ by the definition. 
  \item\label{4} For $l$ from $1$ to $e$, 
    evaluate $\bS (\a^l ;x)$ from $\bS (\a^{l-1} ;x)$, by multiplying 
    $\conti{i}^{\pm 1}$ as Corollary \ref{cor:contiguity}. 
  \item By Proposition \ref{prop:NCtoS} and Corollary \ref{cor:euler}, 
    we obtain the numerical values of $Z$ and $E[U_{ij}]$. 
  \item By the expressions (\ref{eq:psi-expression}) and (\ref{eq:pfaffian-expression}), 
    evaluate $\Psi (\a;x)$ and $\Psi (\a;x) \bS(\a ;x) (=\dx \bS(\a ;x))$. 
  \item By Corollary \ref{cor:euler}, 
    we obtain the numerical values of $\pa E[U_{ij}]/ \pa x_{i'j'}$. 
  \end{enumerate}
\end{algorithm}
\begin{remark}
  Though the evaluation of $\conti{i}^{\pm 1}$ needs 
  the inverse matrices of intersection matrices, 
  these have explicit expression; see Remark \ref{rem:appendix} (1). 
  Thus, it is not hard to evaluate $\conti{i}^{\pm 1}$. 
\end{remark}

\begin{exam}[$r_1=r_2=3$ ($k=n=2$)] 
  We consider $3\times 3$ contingency tables whose marginal sums 
  and probabilities are given as follows, respectively. 
  \begin{align*}
    \begin{array}{|c|c|c|c}
      \cline{1-3} 
      & & &2 \\ \cline{1-3} 
      & & &3 \\ \cline{1-3} 
      & & &3 \\  \cline{1-3} 
      \multicolumn{1}{c}{1}&\multicolumn{1}{c}{3} 
      &\multicolumn{1}{c}{4}  & 8
    \end{array} 
    \qquad 
    p=
    \begin{array}{|c|c|c|}
      \hline 
      1& 1/2& 1/3 \\ \hline 
      1& 1/5& 1/7  \\ \hline 
      1& 1&1  \\  \hline 
    \end{array} 
  \end{align*}
  In this case, 
  the notations appearing in Algorithms \ref{algo:path} and \ref{algo:NC} 
  are as follows: 
  \begin{align*}
    &\a =(-3,-2,-3,3,4,1), \ 
    x_{11}=\frac{1}{2},\ x_{12}=\frac{1}{3}, \ 
    x_{21}=\frac{1}{5},\ x_{22}=\frac{1}{7},  \\
    &\a^0=(-2,-1,-1,1,1,2), \ 
    \si =(-1,-1,-2,2,3,-1), \ 
    e=9. 
  \end{align*}
  We write down the changes of parameters (see Figure \ref{fig:3x3}). 
  \begin{figure}
    \begin{align*}
      \bS(-2,-1,-1,1,1,2&;x)=\bS(\a^0;x), \\
      \frac{1}{2} \conti{3} (-2,-1,1,1,1,2;x) \times \downarrow & \\
      \bS(-3,-1,-1,2,1,2&;x)=\bS(\a^1;x), \quad \a^1 =\a^0 +\de_3, \\
      \frac{1}{3} \conti{3} (-3,-1,-1,2,1,2;x) \times \downarrow & \\
      \bS(-4,-1,-1,3,1,2&;x)=\bS(\a^2;x), \quad \a^2 =\a^1 +\de_3, \\
      \frac{1}{2} \conti{4} (-4,-1,-1,3,1,2;x) \times \downarrow & \\
      \bS(-5,-1,-1,3,2,2&;x)=\bS(\a^3;x), \quad \a^3 =\a^2 +\de_4, \\
      \frac{1}{3} \conti{4} (-5,-1,-1,3,2,2;x) \times \downarrow & \\
      \bS(-6,-1,-1,3,3,2&;x)=\bS(\a^4;x), \quad \a^4 =\a^3 +\de_4, \\
      \frac{1}{4} \conti{4} (-6,-1,-1,3,3,2;x) \times \downarrow & \\
      \bS(-7,-1,-1,3,4,2&;x)=\bS (\a^5 ;x), \quad \a^5 =\a^4 +\de_4, \\
      2 \conti{5}^{-1} (-6,-1,-1,3,4,1;x) \times \downarrow & \\
      \bS(-6,-1,-1,3,4,1&;x)=\bS (\a^6 ;x), \quad \a^6 =\a^5 -\de_5,\\
      \conti{1}^{-1} (-5,-2,-1,3,4,1;x) \times \downarrow & \\
      \bS(-5,-2,-1,3,4,1&;x)=\bS(\a^7;x), \quad \a^7 =\a^6 -\de_1, \\
      \conti{2}^{-1} (-4,-2,-2,3,4,1;x) \times \downarrow & \\
      \bS(-4,-2,-2,3,4,1&;x)=\bS(\a^8;x), \quad \a^8 =\a^7 -\de_2, \\
      \conti{2}^{-1} (-3,-2,-3,3,4,1;x) \times \downarrow & \\
      \bS(-3,-2,-3,3,4,1&;x)=\bS(\a^9;x)=\bS(\a;x), \quad \a^9 =\a^8 -\de_2. 
    \end{align*}
    \caption{Step \ref{4} in Algorithm \ref{algo:NC}.}
    \label{fig:3x3}
  \end{figure}
  Note that the G-M vector $\bS(\a;x)$ is 
  \begin{align*}
    \bS (\a ;x) &=({\rm constant})\cdot \int_{\De \la 125 \ra} \prod_{j=1}^5 L_j^{\a_j} \cdot \tr \bigl( 
    \f \la 012 \ra , \f \la 013 \ra , \f \la 014 \ra ,  
    \f \la 023 \ra , \f \la 024 \ra , \f \la 034 \ra 
    \bigr) \\
    &=\tr \Bigl(
    S(\a;x) ,
    \frac{x_{21}}{\a_3}\cdot \frac{\pa S(\a;x)}{\pa x_{21}}, 
    \frac{x_{22}}{\a_4}\cdot \frac{\pa S(\a;x)}{\pa x_{22}},
    \\ & \qquad \qquad  
    \frac{-x_{11}}{\a_3}\cdot \frac{\pa S(\a;x)}{\pa x_{11}},
    \frac{-x_{12}}{\a_4}\cdot \frac{\pa S(\a;x)}{\pa x_{12}}, 
    \frac{x_{11}x_{22}-x_{12}x_{21}}{\a_3 \a_4}\cdot \frac{\pa^2 S(\a;x)}{\pa x_{11} \pa x_{22}}
    \bigr) .
  \end{align*}
\end{exam}

\begin{exam}[$r_1=r_2=5$ ($k=n=4$)] 
  We consider the following case.
  \begin{align*}    
    \begin{array}{|c|c|c|c|c|c}
      \cline{1-5} 
      & & & & &400 \\ \cline{1-5} 
      & & & & &410 \\ \cline{1-5} 
      & & & & &711 \\ \cline{1-5} 
      & & & & &250 \\ \cline{1-5} 
      & & & & &250 \\ \cline{1-5} 
      \multicolumn{1}{c}{560}
      &\multicolumn{1}{c}{361} 
      &\multicolumn{1}{c}{550}
      &\multicolumn{1}{c}{350} 
      &\multicolumn{1}{c}{200} & 2021
    \end{array} 
  \end{align*}
  Evaluation of the expectations takes 19003 seconds 
  on our implementation 
  (with Risa/Asir on a machine with an Intel Xeon (2.70GHz) and 256G memory). 
\end{exam}

\begin{acknowledgements}
  The authors are grateful to Professor Nobuki Takayama 
  for posing this problem and for helping the first author to implement 
  the algorithms on Risa/Asir.
  This work was supported by JSPS KAKENHI Grant Numbers 25220001. 
\end{acknowledgements}

\appendix

\section{Inverse of intersection matrices}
\label{appendix}
We regard $\a_i$ as an indeterminant, 
and entries of matrices as elements in the rational function field 
$\C (\a)=\C (\a_0 ,\ldots ,\a_{k+n+1})$ with a relation $\sum_{i=0}^{k+n+1}\a_i=0$.  
For $f(\a) \in \C (\a)$, we denote $f(-\a)$ by $f(\a)^{\vee}$. 
For a matrix $A\in M(r,r;\C(\a))$ ($r={\binom{k+n}{k}}$), 
let $A^{\vee}$ be the matrix operated ${}^{\vee}$ on each entry of $A$. 
In this appendix, we show the following proposition. 
\begin{proposition}
  \label{prop:inverse}
  For $p_1 \neq q_1$ and $p_2 \neq q_2$, 
  we put 
  $$
  C_{(p_1 q_1)(p_2 q_2)} (\a):=
  \Bigl( \cI (\f \la I \ra , \f \la J \ra ) 
  \Bigr)_{I\in \under{q_1}{\cJ}{p_1},J\in \under{q_2}{\cJ}{p_2}}
  $$
  whose entries are regarded as rational functions of $\a_i$'s. 
  If none of $a_i \in \C (0\leq i \leq k+n+1)$ is zero, 
  then $C_{(p_1 q_1)(p_2 q_2)} (a_0,\ldots ,a_{k+n+1})$ is well-defined and invertible. 
\end{proposition}
\begin{proof}  
  The well-definedness is clear by Fact \ref{fact:intersection}. 
  We show that the matrix is invertible. 
  For $p \neq q$, there exists an invertible matrix $A_{pq}\in M(r,r;\C(\a))$
  such that 
  $$
  \tr \bigl( \cdots ,\f \la J \ra ,\cdots \bigr)_{J\in \under{q}{\cJ}{p}}
  =A_{pq}
  \tr \bigl( \cdots ,\f \la I \ra ,\cdots \bigr)_{I\in \under{k+n+1}{\cJ}{0}} .
  $$
  Let $\{ I_1,\ldots ,I_r \}$ (resp. $\{ J_1,\ldots ,J_r \}$)
  be the set of subsets of $\{ 0,1,\ldots ,k+n+1 \} -\{ p_1 ,q_1 \}$ (resp. $\{ 0,1,\ldots ,k+n+1 \} -\{ p_2 ,q_2 \}$)
  with cardinality $k$. 
  By Fact \ref{fact:intersection}, we have 
  \begin{align*}
    &C_{(p_1 q_1)(q_1 p_1)} =(\tpi)^k \cdot \diag 
    \left( \frac{1}{\prod_{i\in I_1}\a_i} , \ldots , \frac{1}{\prod_{i\in I_r}\a_j} \right), \\
    & C_{(q_2 p_2)(p_2 q_2)}=(\tpi)^k \cdot \diag 
    \left( \frac{1}{\prod_{j\in J_1}\a_j} , \ldots , \frac{1}{\prod_{j\in J_r}\a_j} \right) .
  \end{align*}
  $\C (\a)$-linearity of the intersection form $\cI$ leads
  \begin{align*}
    C_{(p_1 q_1)(p_2 q_2)} =A_{p_1 q_1} C  \tr A_{p_2 q_2}^{\vee} , 
  \end{align*}
  where $C:=C_{(0,k+n+1)(0,k+n+1)}$. 
  We thus have 
  \begin{align*}
    \left( C_{(p_1 q_1)(p_2 q_2)} \right)^{-1}
    &=(\tr A_{p_2 q_2}^{\vee})^{-1} C^{-1} A_{p_1 q_1}^{-1}
    =C_{(q_2 p_2)(p_2 q_2)}^{-1} A_{q_2 p_2} A_{p_1 q_1}^{-1} \\
    &=C_{(q_2 p_2)(p_2 q_2)}^{-1} A_{q_2 p_2} C \tr A_{q_1 p_1}^{\vee} C_{(p_1 q_1)(q_1 p_1)}^{-1} \\
    &=C_{(q_2 p_2)(p_2 q_2)}^{-1} C_{(q_2 p_2)(q_1 p_1)} C_{(p_1 q_1)(q_1 p_1)}^{-1}. 
  \end{align*}
  This equality holds in $M(r,r;\C (\a))$. 
  If none of $a_i \in \C (0\leq i \leq k+n+1)$ is zero, 
  then $C_{(p_1 q_1)(p_2 q_2)}(a_0,\ldots ,a_{k+n+1}) \in M(r,r;\C)$ is invertible, 
  since the right-hand side is well-defined.  
\end{proof}
\begin{remark}
  \label{rem:appendix}
  \begin{enumerate}[(1)]
  \item This proof gives an explicit expression of the inverse matrix of $C_{(p_1 q_1)(p_2 q_2)}$.  
    It is written as a product of the intersection matrix and diagonal ones. 
  \item The matrices $C(\a)$, $P_i(\a)$, and $Q_i(\a)$ in Theorem \ref{th:contiguity}
    coincide with $C_{(0,k+n+1)(0,k+n+1)}(\a)$, $C_{(0,i)(0,k+n+1)}(\a)$, and 
    $C_{(i,0)(0,k+n+1)}(\a)$, respectively. 
  \end{enumerate}
\end{remark}

\end{document}